\documentclass[a4paper, 11pt, final]{article}

\usepackage{amsthm,amsfonts,amsmath}
\usepackage{mathrsfs}
\usepackage{color}
\usepackage{hyperref}
\usepackage{a4wide}
\usepackage{float,enumerate}

\usepackage{pgf,tikz}
\usetikzlibrary{arrows}
\usetikzlibrary{patterns}

\newtheorem{proposition}{Proposition}[section]
\newtheorem{theorem}[proposition]{Theorem}
\newtheorem{lemma}[proposition]{Lemma}
\newtheorem{corollary}[proposition]{Corollary}

\newtheorem{remark}[proposition]{Remark}
\newenvironment{proofof}[1]{\smallskip\noindent{\textbf{Proof~of~#1.}}%
  \hspace{1pt}}{\hspace{-5pt}{\nobreak\quad\nobreak\hfill\nobreak%
    $\square$\vspace{2pt}\par}\smallskip\goodbreak}

\numberwithin{equation}{section}
\allowdisplaybreaks[4]

\renewcommand{\phi}{\varphi}
\renewcommand{\epsilon}{\varepsilon}
\renewcommand{\theta}{\vartheta}
\renewcommand{\L}[1]{\mathbf{L^#1}}
\newcommand{\Lloc}[1]{\mathbf{L^{#1}_{loc}}}
\newcommand{\C}[1]{\mathbf{C^{#1}}}

\newcommand{\BV}{\mathbf{BV}}
\newcommand{\modulo}[1]{{\left|#1\right|}}
\newcommand{\norma}[1]{{\left\|#1\right\|}}
\newcommand{\reali}{{\mathbb{R}}}
\newcommand{\naturali}{{\mathbb{N}}}

\newcommand{\tv}{{\mathop\mathrm{TV}}}
\renewcommand{\O}{\mathinner{\mathcal{O}(1)}}

\newcommand{\Id}{\mathop{\mathbf{Id}}}

\renewcommand{\d}[1]{\mathinner{\mathrm{d}{#1}}}

\newcommand{\diam}{\mathop{\rm{diam}}}
\renewcommand{\hat}[1]{\widehat{#1}}
\newcommand{\caratt}[1]{\chi_{\strut#1}}

\begin{document}

\title{Conservation Laws with\\ Coinciding Smooth Solutions but\\
  Different Conserved Variables}

\author{Rinaldo M. Colombo$^1$ \and Graziano Guerra$^2$}

\footnotetext[1]{INDAM Unit, University of Brescia,
  Italy. \texttt{rinaldo.colombo@unibs.it}}

\footnotetext[2]{Department of Mathematics and its Applications,
  University of Milano - Bicocca,
  Italy. \texttt{graziano.guerra@unimib.it}}

\maketitle

\begin{abstract}
  \noindent Consider two hyperbolic systems of conservation laws in
  one space dimension with the same eigenvalues and (right)
  eigenvectors. We prove that solutions to Cauchy problems with the
  same initial data differ at third order in the total variation of
  the initial datum. As a first application, relying on the classical
  Glimm--Lax result~\cite{GlimmLax}, we obtain estimates improving
  those in~\cite{SaintRaymond} on the distance between solutions to
  the isentropic and non--isentropic inviscid compressible Euler
  equations, under general equations of state. Further applications
  are to the general scalar case, where rather precise estimates are
  obtained, to an approximation by Di Perna of the $p$-system and to a
  traffic model.

  \medskip


  \medskip

\end{abstract}

\section{Introduction}
\label{sec:Intro}

Consider the following Cauchy Problems for $n\times n$ systems of
conservation laws in one space dimension:
\begin{equation}
  \label{eq:3}
  \left\{
    \begin{array}{l}
      \partial_t g (u) + \partial_x f (u) = 0
      \\
      u (0,x) = u_o (x)
    \end{array}
  \right.
  \qquad\mbox{ and } \qquad
  \left\{
    \begin{array}{l}
      \partial_t \tilde g(u) + \partial_x \tilde f (u) = 0
      \\
      u (0,x) = u_o (x)
    \end{array}
  \right.
\end{equation}
where we assume that $u_o$ varies in a neighborhood of a fixed state
$\bar u$. Clearly, the condition
\begin{equation}
  \label{eq:10}
  \left(D g (u)\right)^{-1} \; Df (u)
  =
  \left(D \tilde g (u)\right)^{-1} \; D \tilde f (u)
\end{equation}
ensures that the two systems~\eqref{eq:3} share the same smooth
solutions. This paper is devoted to estimate the difference between
possibly non smooth solutions to~\eqref{eq:3} yielded by the Standard
Riemann Semigroups (\cite[Chapter~9]{BressanLectureNotes}) generated
by these systems.

A first classical situation is the following. Consider the usual Euler
equations for an inviscid compressible fluid in one space dimension,
which, in Eulerian coordinates, read
\begin{equation}
  \label{eq:12}
  \left\{
    \begin{array}{l}
      \partial_t \rho + \partial_x (\rho \, v) = 0
      \\
      \partial_t (\rho \, v)
      +
      \partial_x \left(\rho \, v^2 + p(\rho, s)\right) =0
      \\
      \partial_t
      \left( \frac{1}{2} \rho \, v^2 + \rho \, e (\rho, s) \right)
      +
      \partial_x
      \left(
      \left(
      \frac{1}{2} \, \rho \, v^2 + \rho \, e (\rho, s) + p (\rho, s)
      \right) v
      \right)
      = 0 \,,
    \end{array}
  \right.
\end{equation}
see~\cite[Formula~(3.3.29)]{DafermosBook}, where $t$ is time, $x$ is
the space coordinate, $\rho$ is the mass density, $v$ the speed,
$p = p (\rho, s)$ the pressure, $e = e (\rho, s)$ and $s$ are the
internal energy and the entropy densities per unit mass. A standard
approximation of~\eqref{eq:12} is the so called isentropic $p$--system
\begin{equation}
  \label{eq:15}
  \left\{
    \begin{array}{l}
      \partial_t \rho + \partial_x (\rho \, v) = 0
      \\
      \partial_t (\rho \, v)
      +
      \partial_x \left(\rho \, v^2 + p (\rho, \bar s)\right) = 0 \,,
    \end{array}
  \right.
\end{equation}
see~\cite[Formula~(7.1.12)]{DafermosBook}, where $\bar s$ is a
constant entropy density. Below, we provide precise estimates on the
distance between solutions to~\eqref{eq:12} and to~\eqref{eq:15},
improving those in~\cite{SaintRaymond}. This result,
Theorem~\ref{thm:2}, actually provides a comparison between solutions
to~\eqref{eq:12} and to
\begin{equation}
  \label{eq:11}
  \left\{
    \begin{array}{l}
      \partial_t \rho + \partial_x (\rho \, v) = 0
      \\
      \partial_t (\rho \, v)
      +
      \partial_x \left(\rho \, v^2 + p(\rho, s)\right) =0
      \\
      \partial_t (\rho \, s) + \partial_x (\rho \, v \, s) = 0 \,.
    \end{array}
  \right.
\end{equation}
Indeed, assigning an initial datum with entropy $\bar s$ constant in
space to~\eqref{eq:11} leads to solutions that solve~\eqref{eq:15},
too, see Lemma~\ref{lem:smooth}. Note that~\eqref{eq:11} shares the
same smooth solutions with~\eqref{eq:12}.

\smallskip

The formulations of~\eqref{eq:12} and~\eqref{eq:11} motivate our
choice of presenting general systems of conservation laws in the
form~\eqref{eq:3}, rather than in the standard form
\begin{equation}
  \label{eq:13}
  \left\{
    \begin{array}{l}
      \partial_t w + \partial_x F (w) = 0
      \\
      w (0,x) = w_o (x)
    \end{array}
  \right.
  \quad \mbox{ and } \quad
  \left\{
    \begin{array}{l}
      \partial_t \tilde w + \partial_x \tilde F (\tilde w) = 0
      \\
      \tilde w (0,x) = \tilde w_o (x) \,.
    \end{array}
  \right.
\end{equation}
Clearly, the connection between~\eqref{eq:3} and~\eqref{eq:13} is
given by
\begin{equation}
  \label{eq:14}
  \begin{array}{rcl}
    w
    & =
    & g (u)
    \\
    \tilde w
    & =
    & {\tilde g} (u)
  \end{array}
  \quad \mbox{ and }\quad
  \begin{array}{rcl}
    F (w)
    & =
    & f\left(g^{-1} (w)\right) \,,
    \\
    \tilde F (\tilde w)
    & =
    & \tilde f\left({\tilde g^{-1}} (\tilde w)\right) \,.
  \end{array}
\end{equation}
Assume that systems~\eqref{eq:13} generate Standard Riemann
Semigroups, see~\cite[Definition~9.1]{BressanLectureNotes},
$\mathbb{S} \colon \reali_+ \times \mathbb{D} \to \mathbb{D}$ and
$\tilde{\mathbb{S}} \colon \reali_+ \times \tilde{\mathbb{D}} \to
\tilde{\mathbb{D}}$.
The distance between the orbits of ${\mathbb{S}}$ and those of
$\tilde{\mathbb{S}}$ is estimated
in~\cite[Theorem~2.1]{BianchiniColombo}, but only when the physical
meanings of the conserved variables are the same, so that
$\mathbb{D} = \tilde{\mathbb{D}}$. However, $\mathbb{D}$ and
$\tilde{\mathbb{D}}$ may well be entirely different since the physical
conserved variables $w = g (u)$ need not be the same as
$\tilde w = \tilde g (u)$. Therefore, below we aim at the comparison
between the semigroups
\begin{equation}
  \label{eq:21}
  S_t = g^{-1} \circ \mathbb{S}_t \circ g
  \quad \mbox{ and } \quad
  \tilde S_t = {\tilde g}^{-1} \circ \tilde{\mathbb{S}}_t \circ {\tilde g}
\end{equation}
describing the evolutions of the \emph{same} physical variables $u$,
but with different conserved quantities $w$ and $\tilde w$. For
instance, we have $u = (\rho, v, s)$ in both~\eqref{eq:12}
and~\eqref{eq:11}, while the conserved variables in the two cases are
different, since $w = (\rho, \rho \, v, \rho \, s)$ and
$\tilde w = (\rho, \rho \, v, \frac{1}{2} \, \rho \, v^2 + \rho \, e
(\rho, s))$.
Clearly, in this case, a direct comparison between $w$ and $\tilde w$
in~\eqref{eq:13} is inappropriate.

\smallskip

If the initial datum $u_o$ has sufficiently small total variation,
then the weak entropy solutions $S_t u_o$ and $\tilde S_t u_o$
to~\eqref{eq:3} are known to exist for all times. We prove below sharp
estimates (see~3.~in Theorem~\ref{thm:1}) that imply the bound
\begin{equation}
  \label{eq:9}
  \norma{S_t u_o - \tilde S_t u_o}_{\L1 (\reali; \reali)}
  \leq
  C \; \tv (u_o)^3 \, t\,.
\end{equation}
In the case of systems admitting a full set of Riemann coordinates,
the above estimate can be improved, so that only the negative total
variation appears on the right hand side, (see~4.~in
Theorem~\ref{thm:1}).

Above, $C$ is a suitable constant dependent on $Df$, $Dg$,
$D\tilde f$, $D\tilde g$. A rather careful computation allows to
express the leading term in $C$ by means of $g$ and $\tilde g$, see
Proposition~\ref{prop:1}.

\smallskip

As anticipated above, the present general result, applied
to~\eqref{eq:12} and~\eqref{eq:11}, allows to improve the estimate
obtained in~\cite{SaintRaymond} on the distance between solutions to
the general inviscid Euler equations~\eqref{eq:12} and to the
isentropic $p$--system~\eqref{eq:15}.

\smallskip

As a further application, we compare the usual $p$-system in Eulerian
coordinates with the analogous system where speed is conserved, see
Section~\ref{sec:GM}.

\smallskip

A specific paragraph is devoted to the scalar case, where rather
precise estimates are available. Indeed, the lower order terms in the
estimate provided by Theorem~\ref{thm:scalar} are third order in the
total variation of the initial datum with a coefficient depending on
the $\C0$ norm of $f'' \tilde g'' - \tilde f'' g''$. This estimate is
a counterpart to~\cite[Theorem~2.6]{BianchiniColombo}.

\medskip

Section~\ref{sec:Main} presents our general result, while applications
to gas dynamics are considered in Sections~\ref{sec:Euler}
and~\ref{sec:GM}. All technical details are deferred to
Section~\ref{sec:Proofs}.

\section{Main Result}
\label{sec:Main}

For the basic theory of 1D systems of conservation laws, we refer
to~\cite{BressanLectureNotes, DafermosBook, SerreBooks}.

Throughout, we fix an open bounded set $\Omega$ in $\reali^n$, with
$n \in \naturali$, $n\geq 1$. The following assumptions on the
functions defining systems~\eqref{eq:3} are of use in the sequel:
\begin{enumerate}[\bf(H1)]
\item The functions $f, g, \tilde f, \tilde g$ are defined in
  $\Omega$, attain values in $\reali^n$ and are smooth.
  \begin{itemize}
  \item The functions $g$ and $\tilde g$ are invertible and admit
    smooth inverses $g^{-1}$ and $\tilde g^{-1}$.
  \item For $u \in \Omega$, the matrixes
    $A (u) = \left(D g (u)\right)^{-1} \; Df (u)$,
    $\tilde A (u) = \left(D \tilde g (u)\right)^{-1} \; D \tilde f
    (u)$
    admit the real eigenvalues $\lambda_1 (u), \ldots, \lambda_n (u)$,
    $\tilde\lambda_1 (u), \ldots, \tilde\lambda_n (u)$ with
    $\lambda_{i-1} (u) < \lambda_i (u)$,
    $\tilde\lambda_{i-1} (u) < \tilde\lambda_i (u)$ for
    $i=2, \ldots, n$, and the right eigenvectors $r_1 (u)$, $\ldots$,
    $r_n (u)$, $\tilde r_1 (u)$, $\ldots$, $\tilde r_n (u)$.
  \item In both systems, each characteristic field is either genuinely
    nonlinear or linearly degenerate,
    see~\cite[Definition~5.2]{BressanLectureNotes}.
  \end{itemize}
\item For all $u \in \Omega$, \eqref{eq:10} holds, namely
  $A (u) = \tilde A (u)$.
\item The integral curves of the right eigenvectors define a full set
  of Riemann coordinates.
\end{enumerate}

\noindent Below, we choose $D \lambda_i (u) \cdot r_i (u) \geq 0$ and
$D \tilde \lambda_i (u) \cdot \tilde r_i (u) \geq 0$ for
$i = 1, \ldots, n$ and for all $u \in \Omega$. When necessary, a
specific normalization of the $r_i$, respectively $\tilde r_i$,
in~\textbf{(H1)} is adopted and, consequently, a particular
parametrization of the Lax curves is selected. Here and in what
follows, we assume that the left eigenvectors of $A$, namely
$l_1, \ldots , l_n$, are normalized so that
\begin{displaymath}
  l_i (u) \cdot r_j (u) =
  \left\{
    \begin{array}{lcr@{\;}c@{\;}l}
      1
      & \mbox{if }
      &i
      &=
      &j
      \\
      0
      & \mbox{if }
      &i
      &\neq
      &j
    \end{array}
  \right.
\end{displaymath}
with $r_j$ as in~\textbf{(H1)}.  Concerning~\textbf{(H3)}, see also
the definition of \emph{rich} systems in~\cite[\S~5.9]{SerreBooks}.

It is well known, see~\cite[chapters~7 and~8]{BressanLectureNotes},
that under assumption~\textbf{(H1)} and by~\eqref{eq:21} both
systems~\eqref{eq:3} generate Standard Riemann Semigroups (SRS) $S$
and $\tilde S$ defined on domains $\mathcal{D}$ and
$\tilde{\mathcal{D}}$ containing all $\Lloc1$ functions with
sufficiently small total variation.

With reference to the Riemann Problems
\begin{equation}
  \label{eq:22}
  \left\{
    \begin{array}{l}
      \partial_t g (u) + \partial_x f (u) = 0
      \\
      u (0,x) =
      \left\{
      \begin{array}{lcr@{\;}c@{\;}l}
        u^l
        & \mbox{if}
        & x
        & <
        & 0
        \\
        u^r
        & \mbox{if}
        & x
        & >
        & 0
      \end{array}
          \right.
    \end{array}
  \right.
  \quad \mbox{ and } \quad
  \left\{
    \begin{array}{l}
      \partial_t \tilde g (u) + \partial_x \tilde f (u) = 0
      \\
      u (0,x) =
      \left\{
      \begin{array}{lcr@{\;}c@{\;}l}
        u^l
        & \mbox{if}
        & x
        & <
        & 0
        \\
        u^r
        & \mbox{if}
        & x
        & >
        & 0
      \end{array}
          \right.
    \end{array}
  \right.
\end{equation}
introduce the notation, see~\cite[Chapter~7]{BressanLectureNotes},
\begin{equation}
  \label{eq:27}
  \!\!\!\!\!
  \begin{array}{@{}r@{\,}c@{\,}l@{\quad}l@{}}
    \sigma
    & \to
    & \psi_j (\sigma) (u)
    & \mbox{$j$-th Lax curve of system~\eqref{eq:3}, left, exiting $u$, }
      j=1, \ldots, n.
    \\
    \tilde \sigma
    & \to
    & \tilde \psi_j (\tilde \sigma) (u)
    & \mbox{$j$-th Lax curve of system~\eqref{eq:3}, right, exiting $u$, }
      j=1, \ldots, n.
    \\
    (\sigma_1, \ldots, \sigma_n)
    & =
    & E (u_l, u_r)
    & \mbox{waves' sizes in the solution to Riemann problem~\eqref{eq:22}, left.}
    \\
    (\tilde \sigma_1, \ldots, \tilde \sigma_n)
    & =
    & \tilde E (u_l, u_r)
    & \mbox{waves' sizes in the solution to Riemann problem~\eqref{eq:22}, right.}
  \end{array}\!\!\!
\end{equation}
We set $\sigma_i = E_i (u_l, u_r)$ and
$\tilde\sigma_i = \tilde E_i (u_l, u_r)$, for $i=1, \ldots, n$.

By possibly changing the values of a function
$u \in \BV (\reali; \Omega)$ at countably many points, we assume that
$u$ is right continuous. The distributional derivative $\mu$ of $u$ is
then a vector measure that can be decomposed into a continuous part
$\mu_c$ and an atomic part
$\mu_a$. Following~\cite[Formula~(4.1)]{ColomboGuerra2008bis}, for
$i = 1, \ldots, n$, we consider the wave measure $\mu_i$ defined by
\begin{equation}
  \label{eq:20}
  \mu_i (B)
  =
  \int_B l_i (u) \, \d{\mu_c}
  +
  \sum_{x \in B} E_i\left(u (x-), u (x+)\right) \,,
\end{equation}
for any Borel set $B \subseteq \reali$. Let $\mu_i^+$ and $\mu_i^-$ be
the positive and negative, respectively, parts of $\mu_i$ and
$\modulo{\mu_i} = \mu_i^+ + \mu_i^-$ be the total variation of
$\mu_i$. When necessary, to remind the connection between the measure
$\mu$ and the function $u$, we denote below the left hand side
in~\eqref{eq:20} by $\mu_i (u, B)$, setting also
$\modulo{\mu_i (u, B)} = \mu_i^+ (u, B) + \mu_i^-(u, B)$.

\begin{theorem}
  \label{thm:1}
  Let $f,\tilde f, g, \tilde g$ satisfy~\textbf{(H1)}. Fix
  $\bar u \in \Omega$. Let $\hat\lambda$ be an upper bound for all
  characteristic speeds of both systems~\eqref{eq:3} and define for
  $a,b \in \reali$ with $a < b$
  \begin{equation}
    \label{eq:47}
    \begin{array}{rcl}
      I_t
      & =
      & [a+\hat\lambda\,t, b-\hat\lambda\, t]
      \\
      \mathcal{T}_t
      & =
      & \left\{
        (\tau,x) \in \reali^+ \times \reali \colon
        \tau \in [0,t] \mbox{ and } x \in I_\tau
        \right\}
    \end{array}
    \mbox{ for } t>0
    \mbox{ and }t < (b-a)/\hat\lambda.
  \end{equation}
  Then,
  \begin{enumerate}
  \item The two systems~\eqref{eq:3} generate the SRSs
    $S \colon \reali_+ \times \mathcal{D} \to \mathcal{D}$ and
    $\tilde S \colon \reali_+ \times \tilde{\mathcal{D}} \to
    \tilde{\mathcal{D}}$.

  \item There exists a positive $\delta$ such that
    \begin{displaymath}
      \left(\mathcal{D} \cap \tilde{\mathcal{D}}\right)
      \supseteq
      \left\{ u \in \Lloc1 (\reali; \Omega) \colon
        \norma{u-\bar u}_{\L\infty (\reali; \reali^n)} < \delta
        \mbox{ and }
        \tv (u) < \delta
      \right\} \,.
    \end{displaymath}

  \item If moreover~\textbf{(H2)} holds, there exists a positive
    constant $C$ such that for all
    $u_o \in (\mathcal{D} \cap \tilde{\mathcal{D}})$,
    \begin{equation}
      \label{eq:18}
      \int_{I_t}
      \norma{(S_t u_o) (x) - (\tilde S_t u_o) (x)} \, \d{x}
      \leq
      C \; t \; \tv (u_o; I_0) \,
      \left(\diam u (\mathcal{T}_t)\right)^2
    \end{equation}
    where $u (t,x) = (S_t u_o) (x)$.

  \item If~\textbf{(H2)} and~\textbf{(H3)} hold, then we have the
    improved estimate
    \begin{equation}
      \label{eq:19}
      \int_{I_t}
      \norma{(S_t u_o) (x) - (\tilde S_t u_o) (x)} \, \d{x}
      \leq
      C \; t \, \left(\sum_{i=1} ^n \mu_i^- (u_o; I_0)\right)
      \left(\diam u (\mathcal{T}_t)\right)^2 .
    \end{equation}
  \end{enumerate}
\end{theorem}

\noindent The proof is deferred to Section~\ref{sec:Proofs}. As a
corollary, since $\diam u (\mathcal{T}_t) \leq \O\, \tv (u_o; I_0)$,
we immediately obtain the following result.

\begin{corollary}
  \label{cor:1}
  Let $f$, $\tilde f$, $g$, $\tilde g$ satisfy~\textbf{(H1)}
  and~\textbf{(H2)}. With the same notation as in Theorem~\ref{thm:1},
  \begin{equation}
    \label{eq:24}
    \int_{I_t}
    \norma{(S_t u_o) (x) - (\tilde S_t u_o) (x)} \, \d{x}
    \leq
    C \; t \; \tv (u_o; I_0)^3
  \end{equation}
\end{corollary}

Throughout this paper, $C$ and $\O$ are constants that depends on
norms of $f$, $g$, $\tilde f$, $\tilde g$ computed on $\Omega$. More
detailed information on the constant $C$ appearing in~\eqref{eq:18},
\eqref{eq:19} and~\eqref{eq:24} are provided by the following
Proposition.

\begin{proposition}
  \label{prop:1}
  Let $f,\tilde f, g, \tilde g$ satisfy~\textbf{(H1)}
  and~\textbf{(H2)}. With the same notation as in Theorem~\ref{thm:1},
  define
  \begin{equation}
    \label{eq:aggiunta}
    \!\!\!\begin{array}{@{}c@{\,}l@{}}
    & \displaystyle
      \Delta \! \left((f,g), (\tilde f, \tilde g)\right)
    \\
     =
    & \displaystyle
      \sup_{u \in \Omega}
      \max_{i=1, \ldots, n}
      \norma{
      \left(D \lambda_i (u) \, r_i (u)\right) \,
      \left[
      \left(D\tilde g (u)\right)^{-1} \, D^2 \tilde g (u)
      - \left(D g (u)\right)^{-1} \, D^2 g (u)
      \right]
      \left(r_i (u), r_i (u)\right)
      }.
    \end{array}
  \end{equation}
  Then, the constant $C$ appearing in~\eqref{eq:18}, \eqref{eq:19}
  and~\eqref{eq:24} satisfies
  \begin{equation}
    \label{eq:43}
    C
    =
    \O
    \left(\Delta \! \left((f,g), (\tilde f, \tilde g)\right)
      +
      \diam u (\mathcal{T}_t)\right) \,.
  \end{equation}
\end{proposition}

\begin{remark}
  \label{rem:Delta}
  {\rm Proposition~\ref{prop:1} implies that if
    $\Delta \! \left((f,g), (\tilde f, \tilde g)\right) = 0$
    in~\eqref{eq:43}, then the bounds~\eqref{eq:18}, \eqref{eq:19}
    and~\eqref{eq:24} provide fourth order estimates.}
\end{remark}

\subsection{The Scalar Case}
\label{subs:Scalar}

Consider the Cauchy problems~\eqref{eq:3} in the scalar ($n = 1$)
case, so that the characteristic speed is
$\lambda (u) = f' (u) / g' (u)$. Now, condition~\textbf{(H2)}
takes the form
\begin{equation}
  \label{eq:36}
  \dfrac{f' (u)}{g' (u)} = \dfrac{\tilde f' (u)}{\tilde g' (u)} \,.
\end{equation}
Rather precise estimates are now available, as shown in the next
result.

\begin{theorem}
  \label{thm:scalar}
  In the scalar case, assume that conditions~\textbf{(H1)}
  and~\textbf{(H2)} hold. Then, for any
  $u_o \in \BV (\reali; \Omega)$,
  \begin{eqnarray*}
    &
    & \int_{I_t}
      \modulo{(S_t u_o) (x) - (\tilde S_t u_o) (x)} \, \d{x}
    \\
    & \leq
    & \dfrac{2}{(\inf_\Omega \modulo{g'})(\inf_\Omega \modulo{\tilde g'})}
      \;
      \norma{f'' \; \tilde g'' - \tilde f'' \; g''}_{\C0 (\Omega; \reali)}
      \tv^- (u_o) \; \left(\diam u_o (\reali)\right)^2 \; t
    \\
    &
    & \quad +
      \O
      \left(
      \norma{f' -\tilde f'}_{\C3 (\Omega; \reali)}
      +
      \norma{g' -\tilde g'}_{\C3 (\Omega; \reali)}
      \right)
      \tv^- (u_o) \; \left(\diam u_o (\reali)\right)^3 \, t
  \end{eqnarray*}
\end{theorem}

\noindent Above, by $\tv^- (u)$ we denote the negative total
variation. The proof is deferred to Section~\ref{sec:Proofs}.

A well known possible application of Theorem~\ref{thm:scalar} is given
by the various versions of Burgers' scalar equation
$\partial_t \left(\dfrac{u^m}{m}\right) + \partial_x
\left(\dfrac{u^{m+1}}{m+1}\right) =0$,
for $m \in \{1, 2, \ldots\}$, see~\cite[Formul\ae~(11.34)
and~(11.35)]{LeVequeBook2002}.


\section{The Isentropic Approximation of Euler Equations}
\label{sec:Euler}

On equations~\eqref{eq:12} and~\eqref{eq:11} we assume throughout
that:
\begin{description}
\item[(e)] The internal energy $e$ is a real valued smooth function
  defined on $\left]0, +\infty\right[ \times \reali$ and satisfies
  \begin{displaymath}
    \partial_s e (\rho, s) >0 \,.
  \end{displaymath}

\item[(p)] The pressure $p$ is a real valued smooth function defined
  on $\left]0, +\infty\right[ \times \reali$ and satisfies
  \begin{displaymath}
    p (\rho, s) = \rho^2 \; \partial_\rho e (\rho,s) \,,
    \qquad
    \partial_\rho p (\rho, s) > 0 \,,
    \qquad
    \partial_{\rho\rho}^2 \left(\rho \, p (\rho, s)\right) >0 \,.
  \end{displaymath}
\end{description}

\noindent Above, condition~\textbf{(e)} ensures that the absolute
temperature $\theta = \partial_s e$ is positive,
see~\cite[Formula~(7.1.10)]{DafermosBook}. In~\textbf{(p)}, the former
condition follows from Gibbs relation,
see~\cite[Formula~(2.5.14)]{DafermosBook}, the second states that
pressure is an increasing function of the density at constant
entropy. From the analytic point of view, \textbf{(e)}
and~\textbf{(p)} ensure that both systems~\eqref{eq:12}
and~\eqref{eq:11} satisfy~\textbf{(H1)} and~\textbf{(H2)}.

\begin{lemma}
  \label{lem:gas}
  Let~\textbf{(e)} and~\textbf{(p)} hold. Then, systems~\eqref{eq:12}
  and~\eqref{eq:11} fit into~\eqref{eq:3} setting
  \begin{displaymath}
    u
    =
    \left[
      \begin{array}{c}
        \rho \\ v \\ s
      \end{array}
    \right]
    \quad
    \begin{array}{r@{\,}c@{\,}lr@{\,}c@{\,}l}
      g (u)
      & =
      & \left[
        \begin{array}{c}
          \rho \\ \rho \, v \\ \rho \, s
        \end{array}
      \right]
      & f (u)
      & =
      & \left[
        \begin{array}{c}
          \rho \, v
          \\
          \rho \, v^2 + p (\rho, s)
          \\\rho \, v \, s
        \end{array}
      \right]
      \\
      \tilde g (u)
      & =
      & \left[
        \begin{array}{c}
          \rho \\ \rho \, v \\ \frac12 \, \rho \, v^2 + \rho\, e (\rho,s)
        \end{array}
      \right]
      &\tilde f (u)
      & =
      & \left[
        \begin{array}{c}
          \rho \, v
          \\
          \rho \, v^2 + p (\rho, s)
          \\
          \left(
          \frac{1}{2} \, \rho \, v^2 + \rho \, e (\rho, s) + p (\rho, s)
          \right) v
        \end{array}
      \right] .
    \end{array}
  \end{displaymath}
  Moreover, conditions~\textbf{(H1)} and~\textbf{(H2)} hold, with
  \begin{displaymath}
    A
    =
    \left[
      \begin{array}{@{}ccc@{}}
        v
        & \rho
        & 0
        \\
        \frac{1}{\rho} \, \partial_\rho p
        & v
        & \frac{1}{\rho} \, \partial_s p
        \\
        0
        & 0
        &v
      \end{array}
    \right]
    \quad
    \begin{array}{@{}r@{\,}c@{\,}l@{}}
      \lambda_1
      & =
      & v - \sqrt{\partial_\rho p}
      \\
      \lambda_2
      & =
      & v
      \\
      \lambda_3
      & =
      & v + \sqrt{\partial_\rho p}
    \end{array}
    \qquad
    r_1 = \left[
      \begin{array}{@{}c@{}}
        -1
        \\
        \frac{\sqrt{\partial_\rho p}}{\rho}
        \\
        0
      \end{array}
    \right]
    \quad
    r_2 = \left[
      \begin{array}{@{}c@{}}
        1
        \\
        0
        \\
        -\frac{\partial_\rho p}{\partial_s p}
      \end{array}
    \right]
    \quad
    r_3 = \left[
      \begin{array}{@{}c@{}}
        1
        \\
        \frac{\sqrt{\partial_\rho p}}{\rho}
        \\
        0
      \end{array}
    \right].
  \end{displaymath}
\end{lemma}

\noindent The proof is obtained through elementary computations.

We now check that solutions to the classical $p$-system also
solve~\eqref{eq:11} as soon as the initial datum has constant entropy.

\begin{lemma}
  \label{lem:smooth}
  Let~\textbf{(e)} and~\textbf{(p)} hold. Fix a constant state
  $(\bar\rho, \bar v, \bar s) \in \reali^+ \times \reali \times
  \reali$.
  Call
  $S^{2\times 2} \colon \reali^+ \times \mathcal{D}^{2\times 2} \to
  \mathcal{D}^{2\times 2}$
  the SRS generated by~\eqref{eq:15} and
  $S^{3\times 3} \colon \reali^+ \times \mathcal{D}^{3\times 3} \to
  \mathcal{D}^{3\times 3}$ the SRS generated by~\eqref{eq:11}, with
  \begin{eqnarray*}
    \mathcal{D}^{2\times2}
    & \supseteq
    & \left\{
      (\rho, v) \in
      \Lloc1 (\reali; \reali^+ \times \reali)
      \colon
      \begin{array}{l}
        \norma{(\rho,v) - (\bar\rho, \bar v)}_{\L\infty (\reali; \reali^2)} < \delta
        \\
        \tv (\rho, v) < \delta
      \end{array}
    \right\}
    \\
    \mathcal{D}^{3\times3}
    & \supseteq
    & \left\{
      (\rho, v, s) \in \Lloc1 (\reali; \reali^+ \times \reali \times \reali)
      \colon
      \begin{array}{l}
        \norma{(\rho,v,s) - (\bar\rho, \bar v, \bar s)}_{\L\infty (\reali; \reali^3)} < \delta
        \\
        \tv (\rho, v, s) < \delta
      \end{array}
    \right\}
  \end{eqnarray*}
  for a positive $\delta$. For any $(\rho_o, v_o)$ such that
  $\norma{(\rho_o,v_o) - (\bar\rho, \bar v)}_{\L\infty (\reali; \reali^2)}
  < \delta$ and $\tv (\rho_o, v_o) < \delta$, then
  \begin{displaymath}
    \begin{array}{r@{\,}c@{\,}l}
      (\rho_o, v_o)
      &\in
      & \mathcal{D}^{2\times2}
      \\
      (\rho_o, v_o, \bar s)
      & \in
      & \mathcal{D}^{3\times3}
    \end{array}
    \quad \mbox{ and } \quad
    \left(S^{2\times2}_t (\rho_o, v_o), \bar s\right)
    =
    S^{3\times3}_t (\rho_o, v_o, \bar s) \,.
  \end{displaymath}
\end{lemma}

We are now ready to estimate the distance between solutions
to~\eqref{eq:12} and~\eqref{eq:11}.

\begin{theorem}
  \label{thm:2}
  Let~\textbf{(e)} and~\textbf{(p)} hold. Fix a constant state
  $(\bar\rho, \bar v, \bar s) \in \reali^+ \times \reali \times
  \reali$.
  Call
  $S^{2\times 2} \colon \reali^+ \times \mathcal{D}^{2\times 2} \to
  \mathcal{D}^{2\times 2}$
  the SRS generated by~\eqref{eq:15} and
  $\tilde S^{3\times 3} \colon \reali^+ \times
  \tilde{\mathcal{D}}^{3\times 3} \to \tilde{\mathcal{D}}^{3\times 3}$
  the SRS generated by~\eqref{eq:12}, with
  \begin{eqnarray*}
    \mathcal{D}^{2\times2}
    & \supseteq
    & \left\{
      (\rho, v) \in
      \Lloc1 (\reali; \reali^+ \times \reali)
      \colon
      \begin{array}{l}
        \norma{(\rho,v) - (\bar\rho, \bar v)}_{\L\infty (\reali; \reali^2)} < \delta
        \\
        \tv (\rho, v) < \delta
      \end{array}
    \right\}
    \\
    \tilde{\mathcal{D}}^{3\times3}
    & \supseteq
    & \left\{
      (\rho, v, s) \in \Lloc1 (\reali; \reali^+ \times \reali \times \reali)
      \colon
      \begin{array}{l}
        \norma{(\rho,v,s) - (\bar\rho, \bar v, \bar s)}_{\L\infty (\reali; \reali^3)} < \delta
        \\
        \tv (\rho, v, s) < \delta
      \end{array}
    \right\}
  \end{eqnarray*}
  for a positive $\delta$. If $(\rho_o, v_o)$ is such that
  $\norma{(\rho_o,v_o) - (\bar\rho, \bar v)}_{\L\infty (\reali;
    \reali^2)} < \delta$
  and $\tv (\rho_o, v_o) < \delta$, then
  $(\rho_o, v_o) \in \mathcal{D}^{2\times2}$,
  $(\rho_o, v_o, \bar s) \in \tilde{\mathcal{D}}^{3\times3}$ and, for
  a suitable positive $C$,
  \begin{equation}
    \label{eq:49}
    \begin{array}{cl}
      & \displaystyle
        \int_{I_t}
        \norma{
        \left(S^{2\times2}_t (\rho_o, v_o) (x), \bar s\right)
        -
        \tilde S^{3\times3}_t (\rho_o, v_o, \bar s) (x)} \d{x}
      \\
      \leq
      & \displaystyle
        C \, t \,
        \left(
        \mu_1^- \left((\rho_o, v_o);I_0\right)
        +
        \mu_2^- \left((\rho_o, v_o);I_0\right)
        \right)
        \left(\diam (\rho_o,v_o) (I_0)\right)^2
    \end{array}
  \end{equation}
  where $(\mu_1,\mu_2)$ are the measures~\eqref{eq:20} referred
  to~\eqref{eq:15}.
\end{theorem}

\noindent The proof is deferred to Section~\ref{sec:Proofs}.  The
present theorem improves the analogous result
in~\cite[Theorem~1.2]{SaintRaymond} in the following aspects:
\begin{enumerate}
\item Only the negative variation is present here, so that the
  estimate~\eqref{eq:49} is optimal whenever no shock arises from the
  initial datum.
\item The diameter of the initial datum in~\eqref{eq:49} provides an
  estimate significantly better than its total variation.
\item Theorem~\ref{thm:2} applies to \emph{any} equation of state
  satisfying~\textbf{(e)} and~\textbf{(p)}.
\end{enumerate}
\noindent With reference to~1.~above, note that~\eqref{eq:49} is
localized, hence the initial datum need not be in $\L1$ and the case
of solutions consisting of only rarefactions is included in
Theorem~\ref{thm:2}.

\section{Speed Conservation vs.~Momentum Conservation}
\label{sec:GM}

In~\cite[Section~5]{DiPernaCPAM}, the following system is considered:
\begin{equation}
  \label{eq:4}
  \left\{
    \begin{array}{l}
      \partial_ t \rho + \partial_x (\rho \, v) =0
      \\
      \partial_t v
      +
      \partial_x \left(\frac12 \, v^2 + P (\rho)\right)
      = 0
    \end{array}
  \right.
  \qquad \mbox{ where } \quad
  P' (\rho) = \dfrac{p' (\rho)}{\rho}
\end{equation}
and $p = p (\rho)$ is the pressure law for a polytropic gas, i.e.,
$p (\rho) = (k^2/\gamma) \, \rho^\gamma$ with $k >0$ and $\gamma > 1$.
Theorem~\ref{thm:1} allows to estimate the difference between
solutions to~\eqref{eq:4} and those to the classical $p$-system
\begin{equation}
  \label{eq:5}
  \left\{
    \begin{array}{l}
      \partial_ t \rho + \partial_x (\rho \, v) =0
      \\
      \partial_t (\rho \, v)
      +
      \partial_x \left(\rho \, v^2 + p (\rho)\right)
      = 0 \,.
    \end{array}
  \right.
\end{equation}
Below, we require the pressure law only to satisfy the standard
assumption
\begin{description}
\item[(P)] $p\in \C2 (\reali^+; \reali^+)$ is such that for all
  $\rho>0$, $p' (\rho) > 0$,
  $\frac{\d{}^2~}{\d{\rho}^2}\left(\rho \, p(\rho)\right) >0$.
\end{description}

\noindent We refer to~\cite{GengZhang2009} for a further result on the
estimates of the difference of solutions to systems of the
type~\eqref{eq:4} and~\eqref{eq:5} in the case, with source terms,
motivated by ducts with slowly varying section.

The comparison between~\eqref{eq:4} and~\eqref{eq:5} fits within the
scope of Theorem~\ref{thm:1}.

\begin{lemma}
  \label{lem:p}
  Let~\textbf{(P)} hold. Then, systems~\eqref{eq:4} and~\eqref{eq:5}
  fit into~\eqref{eq:3} setting
  \begin{equation}
    \label{eq:6}
    u
    =
    \left[
      \begin{array}{c}
        \rho \\ v
      \end{array}
    \right]
    \qquad\quad
    \begin{array}{@{}r@{\,}c@{\,}l@{\quad\qquad}r@{\,}c@{\,}l@{}}
      g (u)
      & =
      &\left[
        \begin{array}{c}
          \rho \\ v
        \end{array}
      \right]
      & f (u)
      & =
      & \left[
        \begin{array}{c}
          \rho \, v \\ \frac12 \, v^2 + P (\rho)
        \end{array}
      \right]_{\vphantom{|}}
      \\
      \tilde g (u)
      & =
      &\left[
        \begin{array}{c}
          \rho \\ \rho \, v
        \end{array}
      \right]
      & \tilde f (u)
      & =
      & \left[
        \begin{array}{c}
          \rho \, v \\ \rho \, v^2 + p (\rho)
        \end{array}
      \right] \,.
    \end{array}
  \end{equation}
  Moreover, conditions~\textbf{(H1)}, \textbf{(H2)} and~\textbf{(H3)}
  hold, with
  \begin{displaymath}
    A = \left[
      \begin{array}{cc}
        v
        & \rho
        \\
        \frac{p' (\rho)}{\rho}
        & v
      \end{array}
    \right]
    \qquad
    \begin{array}{r@{\,}c@{\,}l}
      \lambda_1
      & =
      & v - \sqrt{p' (\rho)}
      \\
      \lambda_2
      & =
      & v + \sqrt{p' (\rho)}
    \end{array}
    \qquad
    r_1 = \left[
      \begin{array}{c}
        - \rho
        \\
        \sqrt{p' (\rho)}
      \end{array}
    \right]
    \quad
    r_2 = \left[
      \begin{array}{c}
        \rho
        \\
        \sqrt{p' (\rho)}
      \end{array}
    \right] \,.
  \end{displaymath}
\end{lemma}

\noindent The proof is immediate and hence omitted.

Theorem~\ref{thm:1}, applied to the case of~\eqref{eq:4}
and~\eqref{eq:5} yields the following estimate.

\begin{corollary}
  \label{cor:p}
  Let~\textbf{(P)} hold. Fix a constant state
  $(\bar\rho, \bar v) \in \reali^+ \times \reali$. Call
  $S\colon \reali^+ \times \mathcal{D} \to \mathcal{D}$ the semigroup
  generated by~\eqref{eq:4} and
  $\tilde S\colon \reali^+ \times \tilde{\mathcal{D}} \to
  \tilde{\mathcal{D}}$ the semigroup generated by~\eqref{eq:5} with
  \begin{displaymath}
    \mathcal{D} \cap \tilde{\mathcal{D}}
    \supseteq
    \left\{
      (\rho, v)
      \in
      \Lloc1 \left(\reali; B\left((\bar\rho, \bar v), \delta\right)\right)
      \colon
      \begin{array}{l}
        \norma{(\rho,v) - (\bar\rho, \bar v)}_{\L\infty (\reali; \reali^2)}
        <
        \delta
        \\
        \tv (\rho, v)
        <
        \delta
      \end{array}
    \right\}
  \end{displaymath}
  for $\delta > 0$. Then, for any $(\rho_o, v_o)$ such that
  $\norma{(\rho_o,v_o) - (\bar\rho, \bar v)}_{\L\infty (\reali;
    \reali^2)} < \delta$ and $\tv (\rho_o, v_o) < \delta$,
  \begin{eqnarray*}
    &
    & \int_{I_t} \norma{
      \left(S_t (\rho_o, v_o)\right) (x)
      -
      \left(\tilde S_t (\rho_o, v_o)\right) (x) }
      \d{x}
    \\
    & \leq
    &
      C \, t
      \left(
      \mu_1^-\left((\rho_o, v_o); I_0\right)
      +
      \mu_2^-\left((\rho_o, v_o); I_0\right)
      \right)
      \left(
      \diam (\rho_o,v_o) (I_0)
      \right)^2
  \end{eqnarray*}
  where $(\mu_1, \mu_2)$ are the measures~\eqref{eq:20} referred to
  system~\eqref{eq:4}.
\end{corollary}

\noindent The proof is slightly simpler than that of
Theorem~\ref{thm:2} and, hence, it is omitted.

\section{A Traffic Model}
\label{subs:Traffic}

As a final example, we consider the traffic
model~\cite[Formula~(1.2)]{ColomboMarcelliniRascle}, which reads
\begin{equation}
  \label{eq:25}
  \left\{
    \begin{array}{l}
      \partial_t \rho
      +
      \partial_x \left(\rho \, w \, \psi (\rho)\right) = 0
      \\
      \partial_t w
      +
      w \, \psi (\rho) \, \partial_x w = 0
    \end{array}
  \right.
\end{equation}
where $\rho$ is the car density and $v = w\, \psi (\rho)$ is the
traffic speed at density $\rho$, the Lagrangian variable $w$
describing the maximal speed of drivers. For any smooth invertible
function function $q = q (w)$, system~\eqref{eq:25} can be put in the
conservative from
\begin{equation}
  \label{eq:26}
  \left\{
    \begin{array}{l}
      \partial_t \rho
      +
      \partial_x \left(\rho \, w \, \psi (\rho)\right) = 0
      \\
      \partial_t \left(\rho \, q (w)\right)
      +
      \partial_x \left(\rho \, w \, \psi (\rho) \, q (w)\right) = 0 \,.
    \end{array}
  \right.
\end{equation}
With the notation in Section~\ref{sec:Main}, we have
\begin{displaymath}
  u = \left[
    \begin{array}{@{\,}c@{\,}}
      \rho \\ w
    \end{array}
  \right]
  \qquad\quad
  g (u) = \left[
    \begin{array}{@{\,}c@{\,}}
      \rho \\ \rho \, q (w)
    \end{array}
  \right]
  \qquad\quad
  f (u) = \left[
    \begin{array}{@{\,}c@{\,}}
      \rho \, w \, \psi (\rho)\\ \rho \, w \, \psi (\rho) \, q (w)
    \end{array}
  \right] \, .
\end{displaymath}
Elementary computations yield
\begin{displaymath}
  A (u)
  =
  \left(Dg (u)\right)^{-1} \; Df (u)
  =
  \left[
    \begin{array}{cc}
      \left(\psi (\rho) + \rho \, \psi' (\rho)\right) \, w
      & \rho \, \psi (\rho)
      \\
      0
      & w \, \psi (\rho)
    \end{array}
  \right] \,.
\end{displaymath}
Note that the matrix $A$ is \emph{independent} of the choice of
$q$. Moreover, Remark~\ref{rem:Delta} applies, coherently with the
fact that all systems of the form~\eqref{eq:26} share the same weak as
well as strong solutions, whatever the function $q$,
see~\cite[Remark~5.3]{ColomboMarcelliniRascle}.


\section{Proofs}
\label{sec:Proofs}

Throughout, by $\O$ we denote a quantity dependent only on norms of
$Df$, $Dg$, $D\tilde f$ and $D\tilde g$ computed on a fixed
neighborhood of $\bar u$ in $\Omega$.

\begin{lemma}
  \label{lem:change}
  Let $f,g$ satisfy~\textbf{(H1)}. Then, the function $F$ defined by
  $F = f \circ g^{-1}$ is smooth and for all $w \in g(\Omega)$ the
  matrix $DF (w)$ admits the eigenvalues
  $\Lambda_1 (w), \ldots, \Lambda_n (w)$ and the eigenvectors
  $R_1 (w), \ldots, R_n (w)$, with
  \begin{displaymath}
    \begin{array}{rcl}
      \Lambda_i (w)
      & =
      & \lambda_i \left(g^{-1} (w)\right)
      \\
      R_i (w)
      & =
      & Dg \left(g^{-1} (w)\right) \, r_i \left(g^{-1} (w)\right)
      \\
      D \Lambda_i (w) \, R_i (w)
      & =
      & D \lambda_i \left(g^{-1} (w)\right) \, r_i\left(g^{-1} (w)\right)
    \end{array}
    \qquad \mbox{ for } i=1, \ldots, n\,.
  \end{displaymath}
\end{lemma}

\noindent The proof is immediate and hence omitted.

\begin{proofof}{Theorem~\ref{thm:1}}
  The proof is divided into several steps. We use throughout the
  notation~\eqref{eq:27}.

  \paragraph{Step~1:} Let~\textbf{(H1)} hold. Fix a positive
  $\epsilon$ and let $u^\epsilon$ be an $\epsilon$--approximate front
  tracking solution in the sense
  of~\cite[Definition~7.1]{BressanLectureNotes} to the the Cauchy
  problem in~\eqref{eq:3}, left. Then, for any $T > 0$ such that
  $a + \hat\lambda \, T < b - \hat\lambda \, T$,
  \begin{equation}
    \label{eq:16}
    \begin{array}{cl}
      & \displaystyle
        \int_{I_T}
        \norma{
        u^\epsilon (T,x)
        -
        \left(\tilde S_T \left(u^\epsilon (0)\right)\right) (x)}
        \d{x}
      \\
      \leq
      & \displaystyle
        \tilde L
        \int_0^T
        \sum_{y \in I_t\cap J_t^\epsilon}
        \int_{-\hat\lambda}^{\hat\lambda}
        \norma{
        U(t, y, \xi)
        -
        \left(\tilde{\mathcal{R}}\left(u (t,y-), u (t, y+)\right)\right) (\xi)
        }
        \d\xi
        \d{t}
    \end{array}
  \end{equation}
  where
  \begin{eqnarray}
    \nonumber
    \tilde L
    & =
    & \L1 \mbox{--Lipschitz constant of } t \to \tilde S_t u_o
    \\
    \nonumber
    J^\epsilon_t
    & =
    & \left\{y \in \reali \colon u^\epsilon (t, y-) \neq u^\epsilon (t, y+)\right\}
    \\
    \nonumber
    U (t, y,\xi)
    & =
    & u^\epsilon (t, y-) \, \caratt{]-\infty, \lambda (t, y)[} (\xi)
      +
      u^\epsilon (t, y+) \, \caratt{]\lambda (t, y), +\infty[} (\xi)
    \\
    \nonumber
    \lambda (t, y)
    & =
    & \mbox{speed of the jump in } u^\epsilon (t)
      \mbox{ at time }t
      \mbox{ at point } y \mbox{ with } y \in J^\epsilon_t
    \\
    \label{eq:tildeRP}
    \xi \to \left(\tilde{\mathcal{R}} (u_l, u_r)\right) (\xi)
    & =
    & \mbox{Lax solution to }
      \left\{
      \begin{array}{l}
        \partial_t \tilde g (u) + \partial_\xi \tilde f (u) =0
        \\
        u (0,\xi) =
        \left\{
        \begin{array}{@{\,}lr@{\,}c@{\,}l}
          u_l
          & \xi
          & <
          & 0
          \\
          u_r
          & \xi
          & >
          & 0
        \end{array}
            \right.
      \end{array}
            \mbox{ at } t=1\,.
            \right.
  \end{eqnarray}

  Indeed, by the finite propagation speed of~\eqref{eq:3}, we can
  apply~\cite[Theorem~2.9]{BressanLectureNotes} on the set
  \begin{displaymath}
    \left\{
      (t, x) \in \reali_+ \times \reali \colon
      t \in [0, (b-a) / \hat\lambda]
      \mbox{ and }
      x \in I_t
    \right\}
  \end{displaymath}
  obtaining
  \begin{eqnarray*}
    &
    & \int_{I_T}
      \norma{
      u^\epsilon (T,x)
      -
      \left(\tilde S_T \left(u^\epsilon (0)\right)\right) (x)}
      \d{x}
    \\
    & \leq
    & \tilde L
      \int_0^T
      \liminf_{h \to 0+} \dfrac{1}{h} \,
      \int_{I_t}
      \norma{
      u^\epsilon (t+h,x)
      -
      \left(\tilde S_h \left(u^\epsilon (t)\right)\right) (x)}
      \d{x} \;
      \d{t}
    \\
    & \leq
    & \tilde L
      \int_0^T
      \sum_{y \in I_t \cap J^\epsilon_t}
      \liminf_{h\to 0+}
      \frac{1}{h}
      \int_{y-\hat \lambda h}^{y+\hat \lambda h}
      \norma{
      u^\epsilon (t+h, x)
      -
      \left(\tilde S_h \left(u^\epsilon (t)\right)\right) (x)}
      \d{x} \;
      \d{t}
    \\
    & \leq
    & \tilde L
      \int_0^T
      \sum_{y \in I_t \cap J^\epsilon_t}
      \liminf_{h\to 0+}
      \frac{1}{h}
      \int_{-\hat \lambda h}^{\hat \lambda h}
      \norma{
      u^\epsilon (t+h,y+\xi)
      -
      \left(\tilde S_h \left(u^\epsilon (t)\right)\right) (y+\xi)}
      \d\xi \;
      \d{t}
  \end{eqnarray*}
  which gives the desired estimate, thanks to the hyperbolic rescaling
  $\left[
    \begin{array}{@{}c@{}}
      h \\ \xi
    \end{array}\right] \to \left[
    \begin{array}{@{}c@{}}
      1 \\ \xi/h
    \end{array}
  \right]$.

  \paragraph{Step~2:} Assume that $u_l = \psi_i (\sigma) (u_r)$ and
  define
  $(\tilde \sigma_1, \ldots, \tilde \sigma_n) = \tilde E (u_l,
  u_r)$. Then,
  \begin{equation}
    \label{eq:2}
    \sum_{j\neq i} \modulo{\tilde \sigma_j}
    +
    \modulo{\sigma - \tilde \sigma_i}
    \leq
    \O
    \norma{\psi_i (\sigma) (u_l) - \tilde \psi_i (\sigma) (u_l)} \,.
  \end{equation}

  Indeed, use the Lipschitz continuity of $\tilde E_j$ and recall that
  $\tilde E_j\left(u_l, \tilde\psi_i (\sigma) (u_l)\right) = 0$ for
  $j \neq i$ and
  $\tilde E_i\left(u_l, \tilde\psi_i (\sigma) (u_l)\right) = \sigma$,
  to estimate
  \begin{eqnarray*}
    \sum_{j\neq i} \modulo{\tilde \sigma_j}
    +
    \modulo{\sigma - \tilde \sigma_i}
    & =
    & \sum_{j=1}^n
      \modulo{
      \tilde E_j\left(u_l, \psi_i (\sigma) (u_l)\right)
      -
      \tilde E_j\left(u_l, \tilde\psi_i (\sigma) (u_l)\right)
      }
    \\
    & \leq
    & \O
      \norma{\psi_i (\sigma) (u_l) - \tilde \psi_i (\sigma) (u_l)} \,.
  \end{eqnarray*}

  \paragraph{Step~3:} Let $\sigma \to S_i (\sigma) (u)$, respectively
  $\sigma \to S_i (\sigma) (u)$, be the $i$--shock curve for
  system~\eqref{eq:3}, left, respectively, right, exiting $u$
  parametrized by $\sigma$. Similarly, $\Lambda_i (u,\sigma)$,
  respectively $\tilde\Lambda_i (u,\sigma)$ is the corresponding
  Rankine--Hugoniot speed. Then, if~\textbf{(H2)} holds, by possibly
  reducing $\Omega$, the quantity
  \begin{equation}
    \label{eq:44}
    \kappa_V
    =
    \sup
    \left\{
      \dfrac{\norma{S_i (\sigma) (u) - \tilde S_i (\sigma) (u)}}{\modulo{\sigma^3}}
      +
      \dfrac{\modulo{\Lambda_i (u,\sigma)-\tilde\Lambda_i (u,\sigma)}}{\sigma^2}
      \colon
      \begin{array}{r@{\,}c@{\,}l@{}}
        u
        & \in
        & V
        \\
        S_i (\sigma) (u)
        & \in
        & V
        \\
        \tilde S_i (\sigma) (u)
        & \in
        & V
        \\
        i
        & \in
        & \{1, \ldots, n\}
      \end{array}
    \right\}
  \end{equation}
  is bounded, where $V$ is any open set with
  $\overline{V} \subset \Omega$. The proof of this boundedness is a
  consequence of Lemma~\ref{lem:nuovo}.


\paragraph{Step~4}
Under assumption~\textbf{(H2)} if the open set $V$ is such that
$V \supseteq u (\reali^+, \reali) \cup \tilde u (\reali^+, \reali)$,
we also have
\begin{equation}
  \label{eq:1}
  \!\!\!
  \int_{-\hat\lambda}^{\hat\lambda}
  \norma{
    U(t, y, \xi)
    -
    \left(\tilde{\mathcal{R}}\left(u (t,y-), u (t, y+)\right)\right) (\xi)
  }
  \d\xi
  \leq
  \left\{
    \begin{array}{@{}l@{\;\;\sigma\mbox{ is }}l@{}}
      \O (\kappa_V \modulo{\sigma}^3 + \epsilon \modulo{\sigma})
      & \mbox{a shock,}
      \\
      \O \, \epsilon \, \sigma
      & \mbox{a rarefaction,}
      \\
      \O \, \sigma
      & \mbox{non--physical.}
    \end{array}
  \right.
\end{equation}

    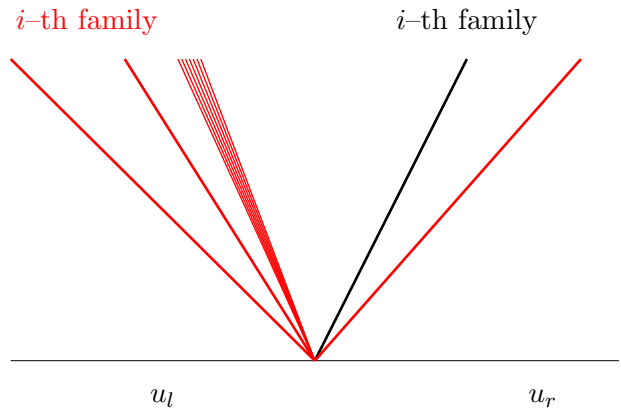
\begin{figure}[!htpb]
      \centering
      \begin{tikzpicture}[line width=0.1mm,yscale=0.5, xscale=0.5]
        \draw[-] (-8,0) -- (8,0); \draw[-,line width=1] (0,0) --
        (4,8); \draw[-,line width=0.5,color=red] (0,0) -- (-3,8);
        \draw[-,line width=0.5,color=red] (0,0) -- (-3.1,8);
        \draw[-,line width=0.5,color=red] (0,0) -- (-3.2,8);
        \draw[-,line width=0.5,color=red] (0,0) -- (-3.3,8);
        \draw[-,line width=0.5,color=red] (0,0) -- (-3.4,8);
        \draw[-,line width=0.5,color=red] (0,0) -- (-3.5,8);
        \draw[-,line width=0.5,color=red] (0,0) -- (-3.6,8);
        \draw[-,line width=1,color=red] (0,0) -- (-5,8); \draw[-,line
        width=1,color=red] (0,0) -- (-8,8); \draw[-,line
        width=1,color=red] (0,0) -- (7,8); \draw (4,9) node
        [label=center:{$i$--th family}] {}; \draw (-6,9) node
        [label=center:{\textcolor{red}{$i$--th family}}] {}; \draw
        (-4,-1) node [label=center:{$u_{l}$}] {}; \draw (6,-1) node
        [label=center:{$u_{r}$}] {};
      \end{tikzpicture}
    \caption{The Riemann problem with data $u_l, u_r$ is solved by a
      single (physical) $i$-wave $\sigma$ of the first system
      in~\eqref{eq:3} and from the waves
      $%
      (\tilde\sigma_1, \ldots,
      \tilde\sigma_n)$
      of the second system in~\eqref{eq:3}. Note that
      $u_r = \psi_{i} (\sigma) (u_{l})%
      $.}
    \label{fig:RP}
  \end{figure}
  Indeed, let $y \in J^\epsilon_t$ and call $u_l = u (t, y-)$ and
  $u_r = u (t, y+)$.

  Assume first that $\sigma \geq 0$, so that
  $\sigma = \O \, \epsilon$. Then, \textbf{(H2)} ensures that
  rarefaction curves of the two systems coincide and hence
  $u_r = \psi_i (\sigma) (u_l) = \tilde\psi_i (\sigma)
  (u_l)$. By~\cite[(ii)~in Lemma~9.1]{BressanLectureNotes},
  \begin{eqnarray*}
    \int_{-\hat\lambda}^{\hat\lambda}
    \norma{
    U(t, y, \xi)
    -
    \left(\tilde{\mathcal{R}} (u_l, u_r)\right) (\xi)
    }
    \d\xi
    & =
    & \int_{-\hat\lambda}^{\hat\lambda}
      \norma{
      U(t, y, \xi)
      -
      \left(\mathcal{R} (u_l, u_r)\right) (\xi)
      }
      \d\xi
    \\
    & \leq
    & \O \, \sigma^2
    \\
    & \leq
    & \O \, \epsilon \, \sigma \,.
  \end{eqnarray*}
  On the other hand, assume $\sigma < 0$.  Applying~\eqref{eq:2}
  and~\eqref{eq:44},
  \begin{equation}
    \label{eq:45}
    \!\!\!\!\!
    \modulo{\tilde\sigma_i - \sigma}
    \leq
    \O \, \norma{\tilde\psi_i (\sigma) (u_l) - \psi_i (\sigma) (u_l)}
    =
    \O \, \norma{\tilde S_i (\sigma) (u_l) - S_i (\sigma) (u_l)}
    \leq
    \O \, \kappa_V \, \modulo{\sigma}^3
  \end{equation}
  which ensures that $\tilde\sigma_i < 0$. Define
  $\lambda_i = \lambda (t,y)$,
  $\tilde u_{i-1} = \tilde\psi_{i-1} (\sigma_{i-1}) \circ \ldots \circ
  \tilde\psi_1 (\sigma_1) (u_l)$
  and $\tilde\lambda_i = \Lambda_i (\tilde u_{i-1}, \tilde\sigma)$.
  Assume that $\tilde\lambda_i \leq \lambda_i$, the other case being
  analogous.
  \begin{equation}
    \label{eq:7}
    \begin{array}{@{}l@{}}
      \displaystyle
      \int_{-\hat\lambda}^{\hat\lambda}
      \norma{
      U(t, y, \xi)
      -
      \tilde{\mathcal{R}}(u_l, u_r) (\xi)
      }
      \d\xi
      \leq
      \int_{-\hat\lambda}^{\tilde\lambda_i}
      \norma{
      u_l
      -
      \tilde{\mathcal{R}}(u_l, u_r) (\xi)
      }
      \d\xi
      \\
      \qquad\qquad\qquad\qquad\qquad\qquad
      \displaystyle
      +
      \int_{\tilde\lambda_i}^{\lambda_i}
      \norma{
      u_l
      -
      \tilde{\mathcal{R}}(u_l, u_r) (\xi)
      }
      \d\xi
      +
      \int_{\lambda_i}^{\hat\lambda}
      \norma{
      u_r
      -
      \tilde{\mathcal{R}}(u_l, u_r) (\xi)
      }
      \d\xi
    \end{array}
  \end{equation}
  Compute the three terms above separately. For
  $\xi < \tilde\lambda_i$, using~\eqref{eq:2} and~\eqref{eq:44},
  \begin{equation}
    \label{eq:46}
    \norma{u_l - \tilde{\mathcal{R}} (u_l, u_r) (\xi)}
    \leq
    \O \, \sum_{j<i} \modulo{\tilde\sigma_j}
    \leq
    \O \, \norma{\tilde S_i (\sigma) (u_l) - S_i (\sigma) (u_l)}
    \leq
    \O \, \kappa_V \, \modulo{\sigma}^3 \,.
  \end{equation}
  As a particular case of the above estimate, note that
  $\norma{u_l - \tilde u_{i-1}} \leq \O \, \kappa_V \,
  \modulo{\sigma}^3$.
  Hence, to estimate the middle summand in the right hand side
  of~\eqref{eq:7}, use \eqref{eq:45} to obtain
  \begin{eqnarray*}
    &
    &\int_{\tilde\lambda_i}^{\lambda_i}
      \norma{
      u_l
      -
      \tilde{\mathcal{R}}(u_l, u_r) (\xi)
      }
      \d\xi
    \\
    & \leq
    & \O \,
      \left(
      \modulo{
      \tilde\Lambda_i (\tilde u_{i-1}, \tilde \sigma_i)
      -
      \Lambda_i (u_l, \sigma)} + \epsilon
      \right) \modulo{\sigma}
    \\
    & \leq
    & \O
      \left(
      \modulo{\tilde\Lambda_i (u_l,\sigma) - \Lambda_i (u_l,\sigma)}
      +
      \modulo{\tilde u_{i-1} - u_l}
      +
      \modulo{\tilde\sigma_i - \sigma}
      +
      \epsilon
      \right)
      \modulo{\sigma}
    \\
    & \leq
    & \O
      \left(
      \kappa_V \, \sigma^2
      +
      \kappa_V \, \modulo{\sigma}^3
      +
      \kappa_V \, \modulo{\sigma}^3
      +
      \epsilon
      \right)
      \modulo{\sigma}
    \\
    & \leq
    & \O \, \left(\kappa_V \, \sigma^2 + \epsilon\right)
      \modulo{\sigma} \,.
  \end{eqnarray*}
  The third summand in~\eqref{eq:7}, for $\xi > \tilde\lambda_i$, is
  treated similarly to~\eqref{eq:46}:
  \begin{displaymath}
    \norma{u_r - \mathcal{R} (u_l, u_r) (\xi)}
    \leq
    \sum_{j>i} \modulo{\tilde\sigma_j}
    \leq
    \O \, \kappa_V \, \modulo{\sigma}^3 \,.
  \end{displaymath}
  Therefore, \eqref{eq:7} yields
  \begin{displaymath}
    \int_{-\hat\lambda}^{\hat\lambda}
    \norma{
      U(t, y, \xi)
      -
      \tilde{\mathcal{R}}(u_l, u_r) (\xi)
    }
    \d\xi
    \leq
    \O \, \left(\kappa_V \, \sigma^2 + \epsilon\right) \modulo{\sigma} \,.
  \end{displaymath}

  Finally, the case of a non--physical wave follows from~\cite[(i)~in
  Lemma~9.1]{BressanLectureNotes}, completing the proof of Step~4.

  \paragraph{Step~5:} The previous steps directly imply that
  \begin{equation}
    \label{eq:17}
    \begin{array}{cl}
      & \displaystyle
        \int_{I_T} \norma{u^\epsilon (T,x) - \tilde S_T u^\epsilon (0,x)}
        \d{x}
      \\
      \leq
      & \displaystyle
        \O \, \int_0^T
        \left(
        \kappa_V
        \sum_{y \in I_t \cap J^\epsilon_t \colon \sigma_y < 0} \modulo{\sigma_y}^3
        +
        \sum_{y \in I_t \cap J^\epsilon_t} \epsilon \, \modulo{\sigma_y}
        +
        \sum_{y \in I_t \cap J^\epsilon_t \colon \sigma_y \in \mathcal{N}\mathcal{P}}
        \modulo{\sigma_y}
        \right)
        \d{t}
    \end{array}
  \end{equation}
  where, as usual, with $\mathcal{N}\mathcal{P}$ we denote the set of
  \emph{non--physical} waves,
  see~\cite[Paragraph~7.1]{BressanLectureNotes}.

  \paragraph{Step~6:} Proof of~3.~in Theorem~\ref{thm:1}.

  Below, we exploit the fact that the total variation of the wave
  front tracking approximate solution at time $t$ is bounded by a
  constant times the total variation of the initial
  datum. By~\eqref{eq:17},
  \begin{eqnarray}
    \nonumber
    &
    & \int_{I_T} \norma{u^\epsilon (T,x) - \tilde S_T u^\epsilon (0,x)}
      \d{x}
    \\
    \nonumber
    & \leq
    & \O
      \int_0^T
      \left(
      \kappa_V
      \sum_{y \in I_t \cap J^\epsilon_t \colon \sigma_y < 0} \modulo{\sigma_y}^3
      +
      \epsilon \, \tv (u^\epsilon (t); I_t)
      + \epsilon
      \right)
      \d{t}
    \\
    \nonumber
    & \leq
    & \O
      \int_0^T
      \left(
      \kappa_V
      \max_{y \in I_t \cap J^\epsilon_t \colon \sigma_y < 0}
      \modulo{\sigma_y}^2
      \sum_{y \in I_t \cap J^\epsilon_t \colon \sigma_y < 0} \modulo{\sigma_y}
      +
      \epsilon \, \tv (u^\epsilon (t); I_t)
      + \epsilon
      \right)
      \d{t}
    \\
    \label{eq:questa}
    & \leq
    & \O
      \int_0^T
      \left(
      \kappa_V \,
      (\diam V)^2
      \sum_{y \in I_t \cap J^\epsilon_t \colon \sigma_y < 0} \modulo{\sigma_y}
      +
      \epsilon \, \tv (u^\epsilon (t); I_t)
      + \epsilon
      \right)
      \d{t}
    \\
    \nonumber
    & \leq
    & \O \int_0^T
      \left(
      \kappa_V \,
      (\diam V)^2 \,
      \tv \left(u^\epsilon (t);I_t\right)
      +
      \epsilon \, \tv (u^\epsilon (t); I_t)
      + \epsilon
      \right)
      \d{t}
    \\
    \nonumber
    & \leq
    & \O  \left(
      \kappa_V \,
      (\diam V)^2 \,
      \tv (u_o;I_0)
      +
      \epsilon \, \tv (u_o; I_0)
      + \epsilon
      \right) \, T
  \end{eqnarray}
  in the limit $\epsilon \to 0$ we obtain~\eqref{eq:18}, thanks to the
  arbitrariness of $V$, provided $V \supset u (\reali^+, \reali)$.

  \paragraph{Step~7:} Proof of~4.~in Theorem~\ref{thm:1}.

  By~\textbf{(H3)}, we may measure sizes of $i$-waves through the
  variation in the $i$-th Riemann coordinate. Introduce the following
  functional defined on the wave front tracking approximate solutions
  $u^\epsilon = u^\epsilon (t,x)$ to~\eqref{eq:3}:
  \begin{equation}
    \label{eq:23}
    \Upsilon^\epsilon (t)
    =
    \sum_{y \in I_t \cap J_t^\epsilon \colon \sigma_y < 0} \modulo{\sigma_y}
    +
    C \sum_{(\sigma_y, \sigma_{y'}) \in \mathcal{A^*} (t)}
    \modulo{\sigma_y \, \sigma_{y'}}
  \end{equation}
  where $\mathcal{A}^* (t)$ is the set of pairs of approaching waves
  in $u^\epsilon$ at time $t$
  (see~\cite[\S~7.3]{BressanLectureNotes}), that we modify excluding
  all pairs of rarefaction waves, also those belonging to different
  families.

  The map $t \to \Upsilon^\epsilon (t)$ is non increasing. Indeed,
  assume that two waves interact at time $\bar t$.  Whenever the
  interacting waves are not both rarefactions, the standard
  interaction estimates apply,
  see~\cite[\S~7.3]{BressanLectureNotes}. In interactions involving
  two rarefactions, \textbf{(H3)} ensures that
  $\Delta \Upsilon^\epsilon\left(u^\epsilon (\bar t)\right) =0$, since
  rarefactions simply cross each other and their sizes measured by
  means of Riemann coordinates remain constant.

  We now have:
  \begin{eqnarray*}
    \sum_{y \in I_t \cap J_t^\epsilon \colon \sigma_y < 0} \modulo{\sigma_y}
    & \leq
    & \Upsilon^\epsilon (t)
      \; \leq \;
      \Upsilon^\epsilon (0)
      \; =
      \sum_{y \in I_{0+} \cap J_{0+}^\epsilon \colon \sigma_y < 0} \modulo{\sigma_y}
      +
      C \sum_{(\sigma_y, \sigma_{y'}) \in \mathcal{A^*} (0+)}
      \modulo{\sigma_y \, \sigma_{y'}}
    \\
    & \leq
    & \left(1 + C \, \tv (u_o, I_0)\right)
      \sum_{y \in I_{0+} \cap J_{0+}^\epsilon \colon \sigma_y < 0} \modulo{\sigma_y} \,.
  \end{eqnarray*}
  Denote by $\mu_i^-$ the negative part of the measure~\eqref{eq:20}
  constructed from the initial datum $u_o$. Similarly, denote by
  $\mu_{i,\epsilon}^-$ the analogous measure constructed from the
  $\epsilon$-approximate piecewise constant initial datum
  $u^\epsilon$. Note that by~\cite[Formula~(4.8) in
  Lemma~4.2]{ColomboGuerra2008bis}, we can choose the piecewise
  constant initial datum $u^\epsilon (0,\cdot)$ such that
  $\mu_{i,\epsilon}^- (I_0) \leq \mu_i^- (I_0) + \epsilon$.
  \begin{eqnarray*}
    \sum_{y \in I_{0+} \cap J_{0+}^\epsilon \colon \sigma_y < 0}
    \modulo{\sigma_y}
    & =
    & \sum_{y \in I_{0} \cap J_{0}^\epsilon \colon \sigma_y < 0}
      \qquad
      \sum_{i=1}^n
      \left[
      E_i\left(u^\epsilon (0,y-), u^\epsilon (0, y+)\right)
      \right]^-
    \\
    & =
    & \sum_{i=1}^n \mu_{i,\epsilon}^- (u^\epsilon (0); I_0)
    \\
    & \leq
    & \sum_{i=1}^n \mu_{i}^- (u_o; I_0) + n\, \epsilon \,.
  \end{eqnarray*}
  Summarizing, starting from~\eqref{eq:questa}, the above inequalities
  yield
  \begin{eqnarray*}
    &
    & \int_{I_T} \norma{u^\epsilon (T,x) - \tilde S_T u^\epsilon (0,x)}
      \d{x}
    \\
    & \leq
    & \O T
      \left(
      \kappa_V
      \left(\sum_{i=1}^n \mu_i^- (u_o; I_0) + n\, \epsilon\right)
      (\diam V)^2
      +
      \epsilon \, \tv (u_o; I_0)
      + \epsilon
      \right)
      \d{t}
  \end{eqnarray*}
  and passing to the limit $\epsilon \to 0$ the proof is completed.
\end{proofof}

\begin{lemma}
  \label{lem:nuovo}
  Under assumptions~\textbf{(H1)} and~\textbf{(H2)}, the following
  bound on $\kappa_V$ as defined in~\eqref{eq:44} holds:
  \begin{equation}
    \label{eq:8}
    \kappa_V
    \leq
    \O \left(\Delta \! \left((f,g), (\tilde f, \tilde g)\right) + \diam V\right)
    \,,
  \end{equation}
  where $V$ is an open subset of $\reali^n$ and
  $\Delta \! \left((f,g), (\tilde f, \tilde g)\right)$ is defined
  in~\eqref{eq:aggiunta}.
\end{lemma}

\begin{proof}
  Using the ideas in~\cite[Theorem~5.2]{BressanLectureNotes}, we
  proceed obtaing higher order estimates.

  Note that at the zero-th order, by~\textbf{(H1)}, we have
  \begin{displaymath}
    S_i (0) (u) - \tilde S_i (0) (u) = 0
    \quad \mbox{ and } \quad
    \Lambda_i (u,0) - \tilde\Lambda_i (u,0) =0 \,.
  \end{displaymath}
  To simplify the notation in the computations below, we keep
  $i \in \{1, \ldots,n\}$ and $u \in \Omega$ fixed and set
  \begin{displaymath}
    \lambda_i (u) \to \lambda, \qquad
    \Lambda_i (u,\sigma) \to \Lambda (\sigma), \qquad
    S_i (\sigma) (u) \to S (\sigma)
  \end{displaymath}
  so that the Rankine--Hugoniot conditions now read
  \begin{equation}
    \label{eq:28}
    \Lambda (\sigma) \,
    \left(g \left(S (\sigma)\right) - g (u)\right)
    =
    f \left(S (\sigma)\right) - f (u) \,.
  \end{equation}
  A derivative in the direction $r = r_i (u)$ of the
  eigenvalues--eigenvector relation
  \begin{displaymath}
    \lambda_i (u) \, Dg (u) \, r_i (u) = Df (u) \, r_i (u)
    \quad \mbox{ shortened as } \quad
    \lambda \, Dg (u) \, r = Df (u) \, r
  \end{displaymath}
  yields:
  \begin{equation}
    \label{eq:29}
    (D\lambda \, r) \, Dg (u) \, r
    +
    \lambda \, D^2g (u) (r, r)
    +
    \lambda \, Dg (u) \, Dr \, r
    =
    D^2f (u) \, (r , r ) + Df (u) \, Dr  \, r
  \end{equation}
  where $D^2g (u) (\, \cdot \, , \, \cdot \, )$ and
  $D^2f (u) (\, \cdot \, , \, \cdot \, )$ are bilinear forms.  A
  further derivative in the direction $r$ yields:
  \begin{equation}
    \label{eq:37}
    \begin{array}{@{}r@{}}
      D^2\lambda (r , r ) \,  Dg (u) \, r
      +
      (D\lambda \, Dr  \, r)  \,  Dg (u) \, r
      +
      2 \, (D\lambda \, r) \, D^2g (u) (r , r )
      +
      2\, (D\lambda \, r) \, Dg (u) \, D r \, r
      \\[1pt]
      +
      \lambda \, D^3 g (u) \, (r , r , r)
      +
      3\, \lambda \, D^2g (u) (Dr \, r , r )
      +
      \lambda \, Dg (u) \, D^2 r  \, (r, r)
      +
      \lambda \, Dg (u) \, D r \, D r \, r
      \\[1pt]
      =
      D^3f (u) \, (r , r , r)
      +
      3\, D^2f (u) (Dr \, r, r)
      +
      Df (u) \, D^2 r \, (r, r)
      +
      Df (u) \, D r \, D r \, r \,,
    \end{array}
  \end{equation}
  here, $D^3 g (u) (\, \cdot \,, \, \cdot \, , \, \cdot \,)$ and
  $D^3 f (u) (\, \cdot \,, \, \cdot \, , \, \cdot \,)$ are trilinear
  forms, while $D^2\lambda ( \,\cdot\, , \,\cdot\, )$ and
  $D^2 r ( \,\cdot\, , \,\cdot\, )$ are bilinear ones.  A first
  differentiation of~\eqref{eq:28} with respect to $\sigma$, setting
  $S = S (\sigma)$, $\Lambda = \Lambda (\sigma)$ and denoting the
  differentiation with respect to $\sigma$ with a dot, yields:
  \begin{displaymath}
    \dot\Lambda \; \left(g (S) - g (u)\right)
    +
    \Lambda \, Dg (S) \, \dot S
    =
    Df (S) \, \dot S
  \end{displaymath}
  Setting $\sigma=0$, we obtain
  \begin{displaymath}
    \Lambda (0) \, Dg (u) \, \dot S (0) = Df (u) \, \dot S (0)
  \end{displaymath}
  which implies that $\dot S (0) = r$ and $\Lambda (0) = \lambda$. The
  same result holds for the \emph{``tilde''} system, hence
  \begin{displaymath}
    \dot S (0) - \dot {\tilde S} (0) = 0 \,.
  \end{displaymath}

  Computing the second derivative of~\eqref{eq:28}, we obtain:
  \begin{multline}
    \label{eq:35}
    \ddot\Lambda \left(g (S) - g (u)\right) + 2 \dot \Lambda \, Dg (S)
    \, \dot S + \Lambda \, D^2g (S) (\dot S, \dot S) + \Lambda \, Dg
    (S) \, \ddot S \\ = D^2f (S) (\dot S, \dot S) + Df (S) \, \ddot S
  \end{multline}
  and setting $\sigma=0$ we obtain
  \begin{equation}
    \label{eq:30}
    2\, \dot\Lambda (0) \, Dg (u) \, r
    +
    \lambda \, D^2g (u) (r,r)
    +
    \lambda \, Dg (u) \, \ddot S (0)
    =
    D^2f (u) (r,r) + Df (u) \ddot S (0) \,.
  \end{equation}
  Subtract now term by term \eqref{eq:29} from~\eqref{eq:30} and
  obtain
  \begin{eqnarray}
    \nonumber
    \left(2\, \dot\Lambda (0) - D\lambda\,r \right) \, Dg (u) \, r
    +
    \lambda\, Dg (u) \left(\ddot S (0) - Dr \, r\right)
    =
    Df (u) \, \left(\ddot S (0) - Dr \, r\right)
    \\
    \label{eq:32}
    (2\, \dot\Lambda (0) - D\lambda\,r) \, r
    =
    \left(Dg^{-1} (u) \, Df (u) - \lambda\right)
    \left(\ddot S (0) - Dr \, r\right)
  \end{eqnarray}
  Multiply now both terms in the latter expression~\eqref{eq:32} by
  the $i$-th left eigenvector $l = l_i$ of
  $A (u) = Dg^{-1} (u) \, Df (u)$ to obtain
  \begin{equation}
    \label{eq:31}
    \dot \Lambda (0) = \dfrac12 \, (D\lambda \, r) \,.
  \end{equation}
  We thus proved that
  \begin{displaymath}
    \dot\Lambda (0) - \dot{\tilde\Lambda} (0) = 0\,.
  \end{displaymath}
  The left hand side in~\eqref{eq:32} vanishes, implying that
  $\ddot S (0) - Dr \, r$ is a right eigenvector of
  $Dg^{-1} (u) \, Df (u)$ corresponding to the eigenvalue $\lambda$,
  so that, for a $\beta \in \reali$,
  \begin{equation}
    \label{eq:33}
    \ddot S (0) - Dr \, r = \beta \, r \,.
  \end{equation}
  We parameterize the Lax curves by means of the arc--length in the
  physical variable $u$, obtaining
  \begin{equation}
    \label{eq:42}
    \begin{array}{rcl}
      \sigma \mbox{ arc--length }
      \Rightarrow
      \norma{\dot S} = 1
      & \Rightarrow
      & \dot S^T \, \ddot S = 0
      \\
      \norma{r} = 1
      & \Rightarrow
      & r^T \, Dr \, r  = 0
    \end{array}
  \end{equation}
  so that multiplying both sides of~\eqref{eq:33} by
  $r^T = \dot S (0)^T$, we obtain $\beta=0$ and hence
  \begin{equation}
    \label{eq:34}
    \ddot S (0) = Dr \, r \,.
  \end{equation}
  Since we expressed $\ddot S$ by means of only the vector field $r$,
  we also obtained
  \begin{displaymath}
    \ddot S (0) - \ddot{\tilde S} (0) = 0\,.
  \end{displaymath}
  Differentiate now~\eqref{eq:35} with respect to $\sigma$:
  \begin{eqnarray*}
    &
    & \dddot\Lambda \left(g (S) - g (u)\right)
      +
      3 \, \ddot \Lambda \, Dg (S) \, \dot S
      +
      3 \, \dot \Lambda \, D^2g (S)  (\dot S, \dot S)
      +
      3 \, \dot \Lambda \, Dg (S) \, \ddot S
    \\
    &
    & +
      \Lambda \, D^3g (S) (\dot S,\dot S,\dot S)
      +
      3\, \Lambda \, D^2g (S) (\dot S, \ddot S)
      +
      \Lambda \, Dg (S) \, \dddot S
    \\
    & =
    & D^3f (S) (\dot S,\dot S,\dot S)
      +
      3 \, D^2f (S) (\dot S, \ddot S)
      +
      Df (S) \, \dddot S
  \end{eqnarray*}
  Compute the above terms in $\sigma = 0$, using~\eqref{eq:31}
  and~\eqref{eq:34}, to obtain
  \begin{eqnarray*}
    &
    & 3 \, \ddot\Lambda (0)\, Dg (u) \, r
      +
      \dfrac32 \, (D\lambda \, r) \, D^2g (u) (r,r)
      +
      \dfrac32 \, (D\lambda \, r) \, Dg (u) \, Dr \, r
    \\
    &
    & + \lambda\, D^3g (u) (r,r,r)
      +
      3 \, \lambda \, D^2g (u) (r, Dr \, r)
      +
      \lambda \, Dg (u) \dddot S (0)
    \\
    & =
    & D^3f (u) (r,r,r)
      +
      3 \, D^2f (u) (r, Dr \, r)
      +
      Df (u) \, \dddot S (0) \,.
  \end{eqnarray*}
  Subtract now the latter relation from~\eqref{eq:37} to obtain
  \begin{displaymath}
    \begin{array}{r}
      \left(
      D^2\lambda (r,r) + (D\lambda \, Dr \, r) - 3 \, \ddot\Lambda (0)
      \right)
      Dg (r) \, r
      +
      \frac12 \, (D\lambda \, r) \, D^2g (u) (r,r)
      \\[1pt]
      +
      \frac12 \, (D\lambda \, r) \, Dg (u) \, Dr \, r
      +
      \lambda \, Dg (u)
      \left(
      D^2r (r,r) + Dr \, Dr \, r - \dddot S (0)
      \right)
      \\[1pt]
      =
      Df (u)
      \left(
      D^2r (r,r) + Dr \, Dr \, r- \dddot S (0)
      \right) \,.
    \end{array}
  \end{displaymath}
  Multiply now on the left by $\left(Dg (u)\right)^{-1}$:
  \begin{equation}
    \label{eq:39}
    \begin{array}{r}
      \left(
      D^2\lambda (r,r) + (D\lambda \, Dr \, r) - 3 \, \ddot\Lambda (0)
      \right)
      r
      +
      \frac12 \, (D\lambda \, r) \, \left(Dg (u)\right)^{-1} \, D^2g (u) (r,r)
      \\[1pt]
      +
      \frac12 \, (D\lambda \, r) \, Dr \, r
      +
      \lambda
      \left(
      D^2r (r,r) + Dr \, Dr \, r - \dddot S (0)
      \right)
      \\[1pt]
      =
      \left(Dg (u)\right)^{-1} \,
      Df (u)
      \left(
      D^2r (r,r) + Dr \, Dr \, r- \dddot S (0)
      \right) \,.
    \end{array}
  \end{equation}
  The same computations leading to~\eqref{eq:39} can now be repeated
  with the \emph{``tilde''} system, yielding an expression analogous
  to~\eqref{eq:39} which, subtracted from~\eqref{eq:39}, yields:
  \begin{equation}
    \label{eq:41}
    \begin{array}{r}
      3\left( \ddot{\tilde\Lambda} (0) - \ddot\Lambda (0) \right) \, r
      -
      \frac12 \, (D\lambda \, r)
      \left(
      \left(D\tilde g (u)\right)^{-1} \, D^2 \tilde g (u)
      -
      \left(D g (u)\right)^{-1} \, D^2 g (u)
      \right) (r, r)
      \\
      =
      \left(
      \lambda\, \Id
      -
      \left(D g (u)\right)^{-1} \, Df (u)
      \right)
      \left(\dddot{\tilde S} (0) - \dddot S (0)\right)
    \end{array}
  \end{equation}
  and multiplying on the left by the $i$-th left eigenvector $l = l_i$ gives
  \begin{displaymath}
    3\left( \ddot{\tilde\Lambda} (0) - \ddot\Lambda (0) \right)
    =
    \frac12 \, (D\lambda \, r) \,
    l \,
    \left(
      \left(D\tilde g (u)\right)^{-1} \, D^2 \tilde g (u)
      -
      \left(D g (u)\right)^{-1} \, D^2 g (u)
    \right) (r, r) \,.
  \end{displaymath}
  This ensures that
  $\modulo{\ddot{\tilde\Lambda} (0) - \ddot\Lambda (0)} = \O \,
  \Delta \! \left((f,g), (\tilde f, \tilde g)\right)$, so that
  \begin{displaymath}
    \dfrac{\modulo{\Lambda_i (u,\sigma)-\tilde\Lambda_i (u,\sigma)}}{\sigma^2}
    \leq
    \O \left(
      \Delta \! \left((f,g), (\tilde f, \tilde g)\right)
      +
      \modulo{\sigma}
    \right)
  \end{displaymath}
  and moreover, by~\eqref{eq:41},
  \begin{displaymath}
    \left(\lambda \Id - \left(D g (u)\right)^{-1} \, Df (u)\right)
    \left(\dddot{\tilde S} (0) - \dddot S (0)\right)
    =
    \O \, \Delta \! \left((f,g), (\tilde f, \tilde g)\right) \,.
  \end{displaymath}
  Write
  $\dddot{\tilde S} (0) - \dddot S (0) = \sum_j \alpha_j \, r_j$.
  Then, multiplying the latter expression above by $l = l_j$ on the
  left, we have, for $j \neq i$,
  \begin{equation}
    \label{eq:48}
    \alpha_j = \O \, \Delta \! \left((f,g), (\tilde f, \tilde g)\right) \,.
  \end{equation}
  On the other hand, by the choice~\eqref{eq:42} of the parameterization
  \begin{displaymath}
    \begin{array}{rcccl}
      \ddot S^T \, \dot S
      & =
      & 0
      & \Rightarrow
      & \dddot S^T \dot S + \ddot S^T \, \ddot S = 0
      \\
      \ddot {\tilde S}^T \, \dot {\tilde S}
      & =
      & 0
      & \Rightarrow
      & \dddot {\tilde S}^T \dot {\tilde S} + \ddot {\tilde S}^T \, \ddot {\tilde S} = 0
    \end{array}
    \Rightarrow
    \left(\dddot S (0) - \dddot{\tilde S} (0)\right)^T \, r = 0
  \end{displaymath}
  which ensures that
  $\left(\sum_{j=i}^n \alpha_j \, r_j \right)^T \, r = 0$ and hence
  \begin{displaymath}
    \alpha_i = - \sum_{j \neq i} \alpha_j \, {r_j}^T r = \O \Delta \!
    \left((f,g), (\tilde f, \tilde g)\right)
  \end{displaymath}
  which, together with~\eqref{eq:48}, ensures that
  \begin{displaymath}
    \dfrac{\norma{S_i (\sigma) (u) - \tilde S_i (\sigma) (u)}}{\modulo{\sigma^3}}
    \leq
    \O \left(
      \Delta \! \left((f,g), (\tilde f, \tilde g)\right)
      +
      \modulo{\sigma}
    \right)
  \end{displaymath}
  completing the proof.
\end{proof}

\begin{proofof}{Proposition~\ref{prop:1}}
  Is a direct consequence of Lemma~\ref{lem:nuovo}.
\end{proofof}

\begin{proofof}{Theorem~\ref{thm:scalar}}
  Note that by~\textbf{(H1)} we can assume that
  $\frac{d~}{d u} \left(\frac{f' (u)}{g' (u)}\right) > 0$.

  We follow the same lines of the proof of Theorem~\ref{thm:1}, using
  as wave front tracking solutions those constructed
  in~\cite[Section~6.1]{BressanLectureNotes}. By~\textbf{(H1)}, if $U$
  is a single shock, respectively a rarefaction, then
  $\tilde {\mathcal{R}} (u_l,u_r)$ also consists of a shock,
  respectively a rarefaction.

  In the scalar case, we have now an estimate different
  from~\eqref{eq:1}. While rarefactions are treated entirely in the
  same way, there are no non--physical waves and in the case of shocks
  a finer estimates is available. Indeed, shock curves in the two
  equations in~\eqref{eq:3} coincide so that $U$ and
  $\tilde {\mathcal{R}} (u_l,u_r)$ differ only in the propagation
  speeds $\lambda$ and $\tilde\lambda$ of the shocks. To simplify the
  notation, set $u = u^\epsilon (t, y-)$ and
  $u+\sigma = u^\epsilon (t, y+)$, so that, by Rankine--Hugoniot
  conditions
  \begin{eqnarray*}
    \int_{-\hat\lambda}^{\hat\lambda}
    \modulo{U (t, y, \xi)
    -
    \left(\mathcal{R}
    \left(
    u^\epsilon (t, y-),
    u^\epsilon (t, y+)
    \right)\right) (\xi)}
    \d\xi
    & \!\! = \!\!
    & \int_{-\hat\lambda}^{\hat\lambda}
      \modulo{U (t, y, \xi)
      -
      \left(\mathcal{R}
      \left(
      u,
      u+\sigma
      \right)\right) (\xi)}
      \d\xi
    \\
    & \!\! \leq \!\!
    & \modulo{\lambda - \tilde \lambda}
      \modulo{\sigma} \,.
  \end{eqnarray*}
  We are thus lead to find a general bound on the quantity
  \begin{eqnarray*}
    &
    &  \modulo{\lambda - \tilde\lambda}
    \\
    & =
    & \modulo{
      \dfrac{f (u+\sigma) - f (u)}{g (u+\sigma) - g (u)}
      -
      \dfrac{\tilde f (u+\sigma) - \tilde f (u)}{\tilde g (u+\sigma) - \tilde g (u)}
      }
    \\
    & =
    & \modulo{
      \dfrac{\left(f (u+\sigma) - f (u)\right)
      \left(\tilde g (u+\sigma) - \tilde g (u)\right)
      -
      \left(\tilde f (u+\sigma) - \tilde f (u)\right)\left(g (u+\sigma) - g (u)\right)
      }{\left(g (u+\sigma) - g (u)\right) \left(\tilde g (u+\sigma) - \tilde g (u)\right)}
      }
    \\
    & \leq
    &
      \dfrac{1}{\sigma^2} \;
      \dfrac{1}{(\inf_\Omega \modulo{g'})(\inf_\Omega \modulo{\tilde g'})}
    \\
    &
    & \quad \times
      \modulo{
      \left(f (u+\sigma) - f (u)\right)
      \left(\tilde g (u+\sigma) - \tilde g (u)\right)
      -
      \left(\tilde f (u+\sigma) - \tilde f (u)\right)
      \left(g (u+\sigma) - g (u)\right)
      }
  \end{eqnarray*}
  Consider now the term in the latter modulus above and compute its
  derivatives using~\eqref{eq:36}:
  \begin{eqnarray*}
    k_u (\sigma)
    & =
    & \left(f (u+\sigma) - f (u)\right)
      \left(\tilde g (u+\sigma) - \tilde g (u)\right)
      -
      \left(\tilde f (u+\sigma) - \tilde f (u)\right)
      \left(g (u+\sigma) - g (u)\right)
    \\
    & =
    &\int_0^\sigma\int_0^\sigma
      \left(
      f' (u+\xi) \, \tilde g' (u+\eta)
      -
      \tilde f' (u+\xi) \, g' (u+\eta)
      \right) \d\xi \d\eta \,.
    \\
    k_u' (\sigma)
    & =
    & \int_0^\sigma
      \left(
      f' (u+\xi) \, \tilde g' (u+\sigma)
      -
      \tilde f' (u+\xi) \, g' (u+\sigma)
      \right) \d\xi
    \\
    &
    &\quad +
      \int_0^\sigma
      \left(
      f' (u+\sigma) \, \tilde g' (u+\eta)
      -
      \tilde f' (u+\sigma) \, g' (u+\eta)
      \right) \d\eta \,.
    \\
    k_u'' (\sigma)
    & =
    & \int_0^\sigma
      \left(
      f' (u+\xi) \, \tilde g'' (u+\sigma)
      -
      \tilde f' (u+\xi) \, g'' (u+\sigma)
      \right) \d\xi
    \\
    &
    &\quad +
      \int_0^\sigma
      \left(
      f'' (u+\sigma) \, \tilde g' (u+\eta)
      -
      \tilde f'' (u+\sigma) \, g' (u+\eta)
      \right) \d\eta \,.
    \\
    k_u''' (\sigma)
    & =
    & \int_0^\sigma
      \left(
      f' (u+\xi) \, \tilde g''' (u+\sigma)
      -
      \tilde f' (u+\xi) \, g''' (u+\sigma)
      \right) \d\xi
    \\
    &
    &\quad +
      \int_0^\sigma
      \left(
      f''' (u+\sigma) \, \tilde g' (u+\eta)
      -
      \tilde f''' (u+\sigma) \, g' (u+\eta)
      \right) \d\eta
    \\
    k_u'''' (\sigma)
    & =
    & \int_0^\sigma
      \left(
      f' (u+\xi) \, \tilde g'''' (u+\sigma)
      -
      \tilde f' (u+\xi) \, g'''' (u+\sigma)
      \right) \d\xi
    \\
    &
    &\quad +
      \int_0^\sigma
      \left(
      f'''' (u+\sigma) \, \tilde g' (u+\eta)
      -
      \tilde f'''' (u+\sigma) \, g' (u+\eta)
      \right) \d\eta
    \\
    &
    &\quad -
      2 \left(
      f'' (u+\sigma) \, \tilde g'' (u+\sigma)
      -
      \tilde f'' (u+\sigma) \, g'' (u+\sigma) \,.
      \right)
  \end{eqnarray*}
  Note that $k (0) = k' (0) = k'' (0) = k''' (0) = 0$, hence a Taylor
  expansion yields
  \begin{eqnarray*}
    \modulo{k_u (\sigma)}
    & \leq
    & \sup_w \sup_s \modulo{k''''_w (s)} \; \sigma^4
      \qquad \mbox{ where}
    \\
    \sup_w \sup_s \modulo{k''''_w (s)}
    & \leq
    & 2
      \norma{f'' \, \tilde g'' - \tilde f'' \, g''}_{\C0 (\Omega;\reali)}
    \\
    &
    & \quad + \O
      \left(
      \norma{f' - \tilde f'}_{\C0 (\Omega;\reali)}
      +
      \norma{g' - \tilde g'}_{\C0 (\Omega;\reali)}
      \right) \,
      \modulo{\sigma}
    \\
    &
    & \quad + \O
      \left(
      \norma{f'''' - \tilde f''''}_{\C0 (\Omega;\reali)}
      +
      \norma{g'''' - \tilde g''''}_{\C0 (\Omega;\reali)}
      \right) \,
      \modulo{\sigma} \,.
  \end{eqnarray*}
  Therefore,
  \begin{eqnarray*}
    &
    &\int_{-\hat\lambda}^{\hat\lambda}
      \modulo{U (t, y, \xi)
      -
      \left(\mathcal{R}
      \left(
      u,
      u+\sigma
      \right)\right) (\xi)}
      \d\xi
    \\
    & \leq
    & \dfrac{2}{(\inf_\Omega \modulo{g'})(\inf_\Omega \modulo{\tilde g'})}
      \,
      \norma{f'' \, \tilde g'' - \tilde f'' \, g''}_{\C0 (\Omega;\reali)}
      \,
      \modulo\sigma^3
    \\
    &
    & \quad + \O
      \left(
      \norma{f' - \tilde f'}_{\C0 (\Omega;\reali)}
      +
      \norma{g' - \tilde g'}_{\C0 (\Omega;\reali)}
      \right) \,
      \modulo{\sigma}^4
    \\
    &
    & \quad + \O
      \left(
      \norma{f'''' - \tilde f''''}_{\C0 (\Omega;\reali)}
      +
      \norma{g'''' - \tilde g''''}_{\C0 (\Omega;\reali)}
      \right) \,
      \modulo{\sigma}^4 \,,
  \end{eqnarray*}
  and the proof is completed as in~\textbf{Step~4} in the proof of
  Theorem~\ref{thm:1} using the Maximum Principle for scalar
  conservation laws.
\end{proofof}

\begin{proofof}{Lemma~\ref{lem:smooth}}
  Lemma~\ref{lem:gas} ensures the existence of the semigroups
  $S^{3\times3}$, while that of $S^{2\times2}$ follows
  from~\textbf{(p)} through well known arguments.

  By the properties of the SRSs~\cite{BressanLectureNotes}, it is
  sufficient to compare the solutions to Riemann problems
  for~\eqref{eq:11}, with constant entropy, and~\eqref{eq:15}. Since a
  constant entropy $\bar s$ factorizes in the third
  equation~\eqref{eq:11}, the Lax curves for the $2\times2$
  system~\eqref{eq:15} are the Lax curves for the $3\times3$
  system~\eqref{eq:11} corresponding to the first and third
  families. Therefore, an entropy solution to the Riemann Problem
  for~\eqref{eq:15} is an ntropy solution to the Riemann Problem
  for~\eqref{eq:11} provided the data has constant entropy $\bar
  s$. This concludes the proof.
\end{proofof}

\begin{proofof}{Theorem~\ref{thm:2}}
  Let $t \to \left(\rho^\epsilon (t), v^\epsilon (t)\right)$ be an
  $\epsilon$--approximate wave front tracking solution
  to~\eqref{eq:15},
  see~\cite[Definition~7.1]{BressanLectureNotes}. Since the $2\times2$
  system~\eqref{eq:11} satisfies~\textbf{(H3)}, we parametrize $i$-Lax
  curve through the variation in the $i$-th Riemann coordinate.

  Then, $t \to \left(\rho^\epsilon (t), v^\epsilon (t), \bar s\right)$
  is an $\epsilon$--approximate wave front tracking solution
  to~\eqref{eq:11}. Follow \textbf{Steps~1--5} in the proof of
  Theorem~\ref{thm:1} comparing
  $t \to \left(\rho^\epsilon (t), v^\epsilon (t), \bar s\right)$ to
  the orbit
  $t \to \tilde{S}^{3\times3}_t \left(\rho^\epsilon (0), v^\epsilon
    (0), \bar s\right)$
  and obtain~\eqref{eq:17}. Apply \textbf{Step~7} to
  system~\eqref{eq:11}. Therefore, the total size of negative waves in
  the $\epsilon$-approximate solution to~\eqref{eq:11} at time $t$ is
  bounded, as $\epsilon\to 0$, by a constant times the total size of
  negative waves in the initial datum, obtaining the estimate
  \begin{eqnarray*}
    &
    &    \int_{I_t}
      \norma{
      \left(S^{2\times2}_t (\rho_o, v_o) (x), \bar s\right)
      -
      \tilde S^{3\times3}_t (\rho_o, v_o, \bar s) (x)} \d{x}
    \\
    &   \leq
    &    C \, t \,
      \left(
      \mu_1^- \left((\rho_o, v_o);I_0\right)
      +
      \mu_2^- \left((\rho_o, v_o);I_0\right)
      \right)
      \left(\diam (\rho,v) (\mathcal{T}_t)\right)^2 \,.
  \end{eqnarray*}
  Recall now~\cite[Theorem~3.12]{BianchiniColomboMonti}, which extends
  the classical result~\cite{GlimmLax}, that ensures the estimate
  \begin{displaymath}
    \diam (\rho,v) (\mathcal{T}_t)
    \leq
    \O \diam (\rho_o,v_o) (I_0) \,,
  \end{displaymath}
  completing the proof.
\end{proofof}

\noindent\textbf{Acknowledgment:} The present work was supported by
the PRIN~2015 project \emph{Hyperbolic Systems of Conservation Laws
  and Fluid Dynamics: Analysis and Applications}, by the GNAMPA~2017
project \emph{Conservation Laws: from Theory to Technology} and by the
\emph{Simons -- Foundation grant 346300} together with the
\emph{Polish Government MNiSW 2015-2019 matching fund}.

\end{document}